\documentclass[12pt]{article}
\usepackage{geometry}

\usepackage{amsthm}
\usepackage{graphicx,algorithm, algorithmic}
\usepackage{caption,hyperref}
\hypersetup{
 colorlinks,linkcolor=red,citecolor=red,urlcolor=red,
}

\usepackage[round,colon,authoryear]{natbib}
\usepackage{bbm}

\newtheorem{remark}{Remark}
\usepackage{bbm}

\usepackage{caption,authblk}
\usepackage{subcaption}

\theoremstyle{plain}
\theoremstyle{plain}
\theoremstyle{plain}
\newtheorem{lemma}{Lemma}
\theoremstyle{plain}
\theoremstyle{plain}

\newcommand{\E}{\mathbb{E}}

\usepackage{smile}
\usepackage{natbib}
\usepackage{multirow}

\usepackage{amsmath,}

\def\tr{\mathop{\text{tr}}\kern.2ex}

\def\E{{\mathbb E}}

\newcommand{\betas}{\beta ^*}

\def\sign{\mathop{\text{sign}}}

\def\supp{\mathop{\text{supp}}}

\def\rank{\mathrm{rank}}
\long\def\comment#1{}

\def\vec{\mathop{\text{vec}}}

\def\nuc{\mathop{\star}}
\def\oper{\mathop{\text{op}}}
\def\fro {\mathop{\text{fro}}}
\def\cS{{\mathcal{S}}}

\def\nuc{\mathop{\star}}

\def\oper{\mathop{\text{op}}}

\newcommand{\bel}{\begin{eqnarray}\label}
\newcommand{\eel}{\end{eqnarray}}
\newcommand{\bes}{\begin{eqnarray*}}
\newcommand{\ees}{\end{eqnarray*}}

\newcommand{\la}{\langle}
\newcommand{\ra}{\rangle}

\def\trace{{\textsc{trace}}}

\def\mat{\mathop{\textrm{Mat}}}

\def\##1\#{\begin{align}#1\end{align}}
\def\$#1\${\begin{align*}#1\end{align*}}
\usepackage{enumitem}

\title{On Stein's Identity and Near-Optimal Estimation in High-dimensional Index Models}
\author[1]{ Zhuoran Yang\thanks{zy6@princeton.edu}}
\author[1]{Krishnakumar Balasubramanian\thanks{kb18@princeton.edu}}
\author[1]{Han Liu\thanks{hanliu@princeton.edu}}
\affil[1]{Department of Operations Research and Financial Engineering, Princeton University}

\begin{document}

\title{On Stein's Identity and Near-Optimal Estimation in High-dimensional Index Models}



\maketitle

\begin{abstract}
We consider estimating the parametric components of semi-parametric multiple index models in a high-dimensional and non-Gaussian setting. Such models form a rich class of non-linear models with applications to signal processing, machine learning and statistics. Our estimators leverage the score function based first and second-order Stein's identities and do not require the covariates to satisfy Gaussian or elliptical symmetry assumptions common in the literature. Moreover, to handle score functions and responses that are heavy-tailed, our estimators are constructed via carefully thresholding their empirical counterparts. We show that our estimator achieves near-optimal statistical rate of convergence in several settings. We supplement our theoretical results via simulation experiments that confirm the theory.
\end{abstract}




\section{Introduction}\label{sec:intro}

Consider the semi-parametric index model relating the response ($Y$) and the covariate ($X$)~by 
\begin{align}\label{eq:mainmim}
Y = f \left( \langle \beta^*_1,X \rangle, \ldots, \langle \beta^*_k,X\rangle\right) + \epsilon,
\end{align}
where $X, \{\beta_\ell ^* \}_{\ell \in [k]} \in \mathbb{R}^d$ and $\epsilon$ is a zero-mean noise that is independent of $X$. Here the vectors $\{\beta^*_\ell\}_{\ell \in [k]}$ are the parametric components and the function $f$ is the nonparametric component or the link function. Such a  model is called as multiple index model (MIM) in the literature. In this work, given $n$ i.i.d samples $\{ X_i, Y_i\}_{i=1}^n$ from the above model, where $n < d$, we are concerned with estimating the parametric components $\{ \beta_\ell ^* \}_{\ell \in [k]}$ when $f$ is unknown. More importantly, we do not impose  the assumption that $X$ is  Gaussian or elliptically symmetric, which is commonly made in the literature. Two important special cases of our model include phase retrieval (in which $k=1$), popular in signal processing, and sufficient dimensionality reduction (in which $k\geq 1$), popular in machine learning and statistics. Motivated by these applications, we make a distinction between the case of $k=1$, which is also called as single index model (SIM), and $k>1$ in the rest of the paper.



Estimating the parametric components $\{ \beta_\ell^*\}_{\ell \in [k]}$ without depending on the exact form of the link function appears naturally in several situations. For example, in phase retrieval~\citep{jaganathan2015phase}, one-bit compressed sensing~\citep{boufounos20081} and sparse generalized linear models~\citep{loh2015regularized}, we are interested in recovering a true parameter  based on structured nonlinear measurements. In sufficient dimensionality reduction, where $k$ is typically a fixed number greater than one, but much less than $d$, we would like to estimate the projection onto the subspace spanned by the parametric components $\{\beta_\ell ^*\}_{\ell \in [k]}$ without depending on the specific form of the function $f$. Furthermore, in deep neural networks (DNN), which are cascades of the MIM, the nonparametric component corresponds to the activation function which is pre-specified and the task is to estimate the parametric components, which are used for prediction in the test stage. Hence, it is crucial to develop estimators for the linear component with  both statistical  accuracy and computational efficiency for a wide class of link functions. 




 Several subtle issues arise when we consider optimal estimation in SIM and MIM. Specifically, most existing results depend crucially on the assumption made on $X$ or $f$ and fail to hold when those assumptions are relaxed. Such issues arise even in low-dimensional settings, where $n>d$. Consider, for example, the case of $k=1$ and a known link function  $f(u) = u^2$. This corresponds to phase retrieval, which is a challenging inverse problem that has regained interest in the last few years along with the success of compressed sensing.  A straightforward way to estimate $\beta^*$ is to do nonlinear least squares regression~\citep{lecue2015minimax}, which is a  nonconvex optimization problem. \cite{candes2013phaselift} propose an estimator based on convex relaxations. Although their estimator is optimal when $X$ is sub-Gaussian, they are not agnostic to the link function, i.e., the same result does not hold if the link function is misspecified. 

Direct optimization of the nonconvex phase retrieval problem was considered by \cite{candes2015phasewf} and  \cite{sun2016geometric}, which propose estimators based on iterative algorithms that are statistical optimal. 
However, they rely on the assumption that $X$ is Gaussian. A careful look at their proofs reveal that extending them to a wider class of distributions is significantly challenging -- for example, they require sharp concentration inequalities for polynomials of degree four of $X$,  which would lead to suboptimal rate  even when   $X$ is sub-Gaussian. Furthermore, their results are not agnostic to the link function as well. Similar observations could be made for both convex~\citep{li2013sparse} and nonconvex estimators~\citep{cai2015optimal} for sparse phase retrieval in high dimensions. In addition, a surprising result for SIM was established in~\cite{plan2015generalized}. They show that when $X$ is Gaussian, for a class of unknown link functions, one could estimate $\beta^*$ at the optimal statistical rate with the convex Lasso estimator. Unfortunately, their assumption on the link function is rather restrictive and rule out several interesting models including phase retrieval. Furthermore, none of the above procedures are applicable to the case of MIMs.

\subsection{Motivation}
Our work is primarily motivated by the interesting phenomenon illustrated in~\citep{plan2015generalized} for a class of high-dimensional SIM. Below, we first briefly summarize the result from~\citep{plan2015generalized} and then provide our~\emph{alternative justification} for the same result via Stein's identity. We mainly leverage this alternative justification and propose our estimators for the more general setting we consider. Assuming,   for simplicity, we work in the one-dimensional setting and   are given $n$ i.i.d. samples from the SIM. Consider the least-squares estimator
\begin{align*}
\hat \beta_{LS} = \underset{\beta \in \mathbb{R}}{\argmin}~~ \frac{1}{n}\sum_{i=1}^n \left(Y_i - X_i\beta \right)^2.
\end{align*}
Note that the above estimator is the standard least-squares estimator assuming a linear model (i.e., identity link function). The surprising observation from~\citep{plan2015generalized} is that, under the \emph{crucial} assumption that $X$ is standard Gaussian, $\hat \beta _{LS}$ is a good   estimator  of  $\beta^*$ (up to a scaling) even when the data is generated from a nonlinear SIM. The same holds true for the high-dimensional setting when the minimization is performed in an appropriately constrained norm-ball (for example, the $\ell_1$-ball). Hence the theory developed for the linear setting could be leveraged to understand the performance in the SIM setting. Below, we give an alternative justification for the above estimator as an implication of   Stein's identity in the Gaussian case, which is summarized as follows.
\begin{proposition}[Gaussian Stein's Identity \citep{stein1972bound}]
Let $ X \sim N(0,1)$ and $g  : \mathbb{R} \to \mathbb{R}$ be a continuous function such that  $\E|g'(X)| \leq \infty$. Then we have $ \E[ g(X) X ]  = \E[ g'(X)]$.
\end{proposition}
Note that in our context, if we let $g(X) = f(\langle X,\beta \rangle)$, then we have $\E[ g'(X )] \propto \beta^* $ and $\EE[  f(X) X ]  = \EE [Y \cdot  X]  $. Now consider the following estimator, which is based on performing least-squares on the sample version of the above proposition:
\begin{align*}
\hat \beta_{SL} =\underset{\beta \in \mathbb{R}}{\argmin} ~~\frac{1}{n}\sum_{i=1}^n (Y_i X_i -\beta)^2
\end{align*}
Note that $\hat \beta_{LS}$ and $\hat \beta_{SL}$ are the same estimators assuming $X \sim N(0,1)$, as $n \to \infty$. This observation leads to an alternative interpretation of the estimator proposed by~\citep{plan2015generalized} via Stein's identity for Gaussian random variables. Thus it provides an alternative justification for why the linear least-squares estimator should work in the SIM setting.  Interestingly, a similar procedure based on second-order Stein's identity (see \S\ref{sec:model} for precise definitions) was used in~\cite{candes2015phasewf} to provide a favorable initializer for their gradient descent algorithm for phase retrieval. Our observation also provides an alternative interpretation of the initialization method used in~\cite{candes2015phasewf} by appealing to Stein's identity. These observations also naturally leads to leveraging non-Gaussian versions of Stein's identity for dealing with non-Gaussian covariates. Our estimators based on this motivation is described in detail in \S\ref{sec:firstordertheory} and \S\ref{sec:secondordertheory}. 


\subsection{Related Work}

The success of Lasso and related linear estimators in high-dimensions~\citep{buhlmann2011statistics}, also enabled the exploration of high-dimensional SIMs. Although, this is very much work in progress. As mentioned previously,~\cite{plan2015generalized} show that the Lasso estimator works for the SIMs in   high dimensions when the data is Gaussian. A more tighter albeit an asymptotic results under the same setting was proved in~\cite{thrampoulidis2015lasso}.  Very recently~\cite{goldstein2016structured} extend the results of~\cite{li1989regression} to the high dimensional setting but it suffers from similar problems as mentioned in the low-dimensional setting. \cite{neykov2016agnostic} considered a misspecified phase retrieval model with Gaussian covariates and established rates of convergence. For the case of monotone nonparametric component,~\cite{yang2015sparse} analyze a non-convex least squares approach under the assumption that the data is sub-Gaussian. However, the success of their method hinges on the knowledge of the link function.  Furthermore,~\cite{jiang2014variable, lin2015consistency,zhu2006sliced} analyze the sliced inverse regression estimator in the high-dimensional setting concentrating mainly on support recovery and consistency properties. Similar to the low-dimensional case, the assumptions made on the covariate distribution restrict them from several real-world applications involving non-Gaussian or non-symmetric covariate, for example high-dimensional problems in economics~\citep{fan2011sparse}. Furthermore, several results are established on a case-by-case basis for fixed link function. Specifically~\cite{boufounos20081, ai2014one} and~\cite{davenport20141} consider 1-bit compressed sensing and matrix completion respectively, where the link is assumed to be the sign function. Also,~\cite{waldspurger2015phase} and~\cite{cai2015optimal} propose and analyze convex and non-convex estimators for phase retrieval respectively, in which the link is the square function. All the above works, except~\cite{ai2014one} make Gaussian assumptions on the data and are specialized for the specific link functions. The non-asymptotic result obtained in~\cite{ai2014one} is under sub-Gaussian assumptions, but the estimator is not consistent. Finally, there is a line of work focusing on estimating both the parametric and the nonparametric component~\cite{kalai2009isotron, kakade2011efficient, alquier2013sparse, radchenko2015high}. We do not focus on this situation in this paper as mentioned before.

For multiple index models, relatively less work exist in the high-dimensional setting. In the low-dimensional setting, a line of work for estimation in MIMs is proposed by Ker-Chau Li, which include  inverse regression~\citep{li1991sliced}, principal Hessian directions~\citep{li1992principal} and regression under link violation~\citep{li1989regression}. The proposed estimators are applicable for a class of unknown link functions under the assumption that the covariate follows a Gaussian or symmetric elliptical distribution. Such an assumption is restrictive as often times the covariates are heavy-tailed or skewed~\citep{horowitz2009semiparametric, fan2011sparse}.  
Furthermore, they concentrate only on the low-dimensional setting establishing asymptotic statements. Estimation in high-dimensional MIM under the subspace sparsity assumption was considered in~\cite{chen2010coordinate}, where the results are asymptotic and the  proposed  estimators are not computable in polynomial time.

To summarize, all the above works require restrictive assumption on either the data distribution or on the link function. We propose and analyze an estimator for a class of (unknown) link functions for the case when the covariates are drawn from a non-Gaussian distribution -- under the assumption that we know the distribution \emph{a priori}.  Note that in several situations, one could fit specialized distributions, to real-world data that is often times skewed and heavy-tailed, so that it provides a good generative model of the data. Also, mixture of Gaussian distribution, with the number of components selected appropriately, approximates the set of all square integrable distributions to arbitrary accuracy (see for example~\cite{mclachlan2004finite}). Furthermore, since this is a density estimation problem it is unlabeled and there is no issue of label scarcity. Hence it is possible to get accurate estimate of the distribution in most situations of interest. Thus our work is complementary to the existing literature and provides an estimator for a class of models that is not addressed in the previous works.  

\subsection{Contributions}

As discussed before, there are several subtleties based on the interplay between the assumptions made on $X$ and $f$ when dealing with estimation in SIM and MIM. Thus an interesting question is, whether it is possible to estimate the linear components in SIMs and MIMs with milder assumptions on both $X$ and $f$ in the high-dimensional setting. In this work, we provide a partial answer to this question. We construct estimators that work for a wide class of link functions, including the phase retrieval link function, and for a large family of distributions of $X$, which is assumed to be known \emph{a priori}. We particularly focus on the case when $X$ follows a non-Gaussian distribution that need not be elliptically symmetric or sub-Gaussian, thus making our  method applicable to several situations not possible before. Our estimators are based on Stein's identity for non-Gaussian distributions, which utilizes the score function. Estimating with the score function is challenging due to their heavy tails.  In order to illustrate that, consider the univariate histograms provided in Figure-\ref{fig:gammascore}. The dark shaded, more concentrated one corresponds to the histogram of $10000$ samples from Gamma distribution with   shape and scale  parameters set to $5$ and $0.2$ respectively. The transparent histogram corresponds to the distribution of the score function of the same Gamma distribution. Note that even when the actual Gamma distribution is well concentrated, the distribution is the corresponding score function is well-spread and heavy-tailed. In the high dimensional setting, in order to estimate with the score functions, we require certain vectors or matrices based on the score functions to be well-concentrated in appropriate norms. In order to achieve that, we construct robust estimators via careful truncation arguments to balance the bias (due to thresholding)-variance (of the estimator) tradeoff and achieve the required concentration. In summary, our contribution are as follows:
\begin{itemize}
\item We construct estimators for the parametric component of a sparse SIM and MIM for a class of unknown link function under the assumption that the covariate distribution is non-Gaussian but known a priori. Our results are applicable for the case of vector, matrix or tensor valued covariates with appropriately defined structures to facilitate high-dimensional estimation. 
\item  We establish near-optimal statistical rates for our estimators. Our results complement the existing ones in the literature and hold in several case not possible before.
\item We provide alternative justifications based on  the  Stein's identity for the estimator used in~\cite{plan2015generalized} for sparse SIM and the initializer used in~\cite{candes2015phasewf} for phase retrieval. 
\item As a consequence of our results for SIM and MIM, we also obtain a near-optimal estimator for sparse PCA with heavy-tailed data in the moderate sample size regime.
\item We provide numerical simulations that confirm our theoretical results.
\end{itemize}

Parts of the results presented in this work, appeared in~\cite{yang2017high} and~\cite{yang2017estimating} previously.  

\begin{figure}[t]
\begin{center}
\includegraphics[scale=0.4]{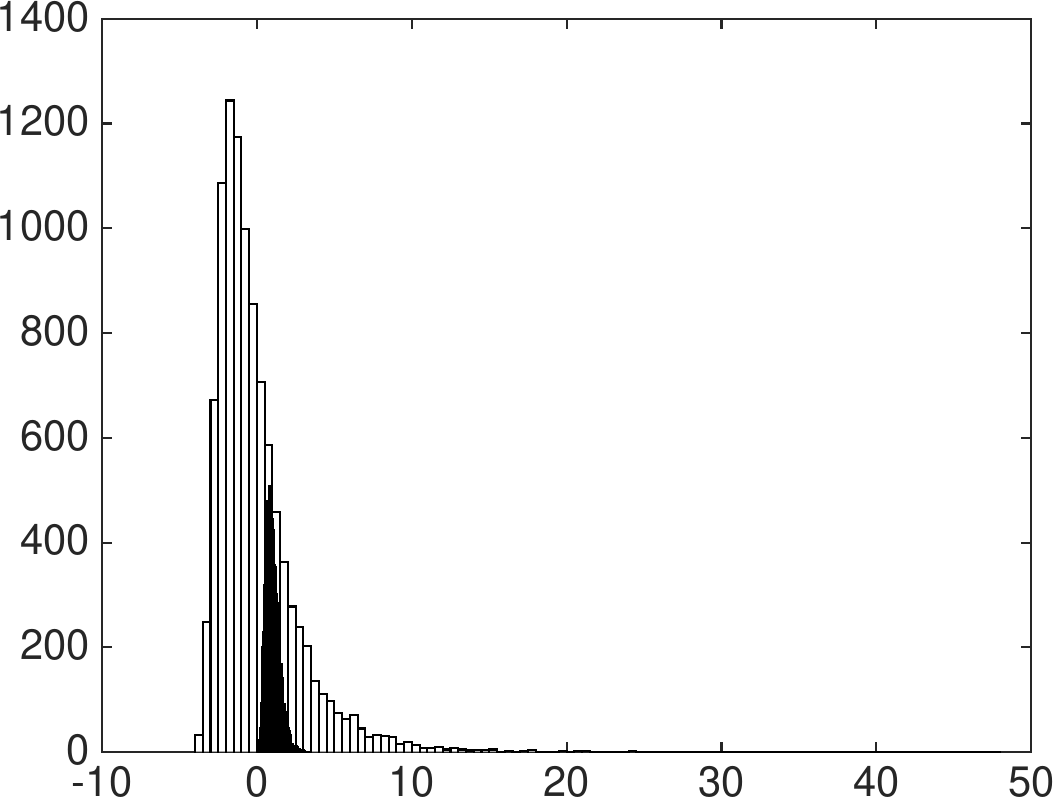}
\end{center}
\caption{Histogram of Score Function based on $10000$ independent samples from the Gamma distribution with shape $5$ and scale  $0.2$.  The dark histogram (we recommend the reader to zoom in to notice it) concentrated around zero corresponds to the Gamma distribution and the transparent histogram corresponds to the distribution of the score of the same Gamma distribution.}
\label{fig:gammascore}
\end{figure}

\subsection{Notations}
In this section, we introduce the notation and define the single index models. Throughout this work, we use $[n]$ to denote the set $\{ 1, \ldots, n\}$. In addition, for a vector $v \in \RR^d$, we denote by $\| v \|_{p}$ the $\ell_p$-norm of $v$ for any $p \geq 1$. We use $\cS^{d-1}$ to denote the unit sphere in $\RR^d$, which is defined as $\cS ^{d-1} = \{ v \in \RR^d \colon \| v \|_2 = 1 \}$. In addition, we define  the support of $v \in \RR^d$  as $\supp(v) = \{ j \in [d], v_j \neq 0 \}$.    Moreover, we denote the nuclear norm, operator norm, element-wise max norm and Frobenius norm of a matrix  $A\in \RR^{d_1 \times d_2}$ by $\| \cdot \|_{\nuc}$, $\| \cdot \|_{\oper}$, $\| \cdot \|_{\infty}$ and $\| \cdot \|_{\fro}$, respectively.   We denote by $\vec(A)$ the vectorization of matrix $A$, which is a vector in $\RR^{d_1\cdot d_2}$. For two matrices $A,B \in \mathbb{R}^{d_1 \times d_2}$ we define the trace inner product as $\la A,B \ra = \trace(A^\top B)$. Note that  it can be viewed as the  standard inner product between $\vec(A)$ and $\vec(B)$. In addition, for an univariate function $g \colon \RR \rightarrow \RR$, we denote by $g \circ   (v)$ and $g\circ(A)$ the output of applying  $g$  to each element of  a vector $v$ and a matrix $A$, respectively. Finally, for a random variable $X \in \RR$ with density $p$, we use $p^{\otimes d} \colon \RR^d \rightarrow \RR$ to denote the joint density of $\{X_1, \cdots, X_d\}$, which are $d$ identical copies of $X$. We also require some notations about tensors. We concentrate on fourth-order tensors for simplicity. For any  fourth-order tensor $ Z \in \RR^{d_1 \times d_2 \times d_3 \times d_4 }$, we denote its $(j_1,j_2, j_3, j_4)$-th entry by $Z (j_1, j_2, j_3, j_4)$. If $d_{\ell} = d $ for all $\ell \in [4]$, we denote the tensor as $Z \in \RR^{d ^\otimes 4}$. Similar to the matrix case, we define $\vec(Z) \in \RR^{d^4}$ as the vectorization of the tensor $Z$.  For two tensors $W,  Z \in \RR^{d^\otimes 4} $, we define their inner inner product  as
\#\label{eq:tensor_inner}
 \la Z, W \ra & = \vec(Z)^\top \vec(W) \\ & = \sum_{j_1 , j_2, j_3, j_4 \in [d]} Z(j_1, j_2, j_3, j_4) \cdot W(j_1, j_2, j_3, j_4) \nonumber
\#
The tensor Frobenius norm of  is also denoted by $\| \cdot  \|_{\fro} $.



\section{Index Models} \label{sec:model}
Now we are ready to define the precise statistical models that we consider in this work. As mentioned above, we consider the case of $k=1$ (SIM) and $k>1$ (MIM) separately. We primarily distinguish our models based on the assumption made on the link functions. We also require the following definition of score function of random variable.  Let $p \colon \RR^d \rightarrow \RR$ be a probability density function defined on $\RR^d$. The score function $S_{p} \colon \RR^d \rightarrow \RR$ associated to density $p$ is defined as  
\$
S_p(x) = - \nabla_x  [\log p(x) ] = - \nabla _x p(x) / p(x).
\$
Note that in the above definition,  the derivative is  taken with respect to $x$. This is different from the more traditional definition of the score function where the density belongs to a parametrized family and the derivative is taken with respect to the parameters. In the rest of the paper to simplify the notation, we omit the subscript $x$ from $\nabla_x$. We also omit the subscript $p$ from $S_p$ when the underlying density $p$ is clear from the context.

\subsection{First-order Link Functions}\label{sec:firstorder}
We first discuss a class of SIM that are based on a certain first-order link functions. We discuss the motivation for our estimator, which automatically highlights the first-order assumption on the link function as well. Recall that our estimators are based on Stein's identity. To begin with, we present the first-order non-Gaussian Stein's identity.

\begin{proposition}[First-order Stein's Identity~\citep{stein2004use}] \label{prop:1stein}
Let $ X \in \RR^d$ be a  real-valued random vector with density $p$. Assume that $p\colon \RR^d \rightarrow \RR$ is differentiable. In addition, let  $g : \mathbb{R}^d \to \mathbb{R}$ be a continuous function such that  $\EE [ \nabla g (X)]  $ exists. Then it holds that 
$$ \EE [ g(X) \cdot S(X) ]  = \EE [ \nabla g(X) ] ,$$
where $S(x) = -  \nabla p(x) / p(x)$ is the score function of $p$.
\end{proposition}

One could apply the above Stein's Identity to SIMs to obtain an estimate of $\beta^*$. To see this, note that when $X \sim N(0, I_d)$ we have $S(x) = x$,  $\forall x \in \RR^d$.  In this case, as $\EE(\epsilon) =0$,  we have $$ \EE(Y \cdot X ) = \EE [ f( \la X, \beta^* \ra) \cdot X] = \EE[ f'( \la X, \betas \ra ) ] \cdot  \betas .$$
Hence one could estimate $\beta^*$ based on estimating the moment $\EE (Y \cdot X)$. This  observation leads to the  estimator   proposed in~\cite{plan2015generalized}. This motivates the following definition of SIM with first order link functions. 

\begin{definition}[Vector SIM with First-order Links] \label{def:simmodel}
Under this model, we assume that the response variable $Y \in \RR$ and the covariate $X\in \RR^d$ are linked via
\#\label{eq:simmodel}
Y = f  (   \la X,  \betas \ra ) + \epsilon,
\#
where $f\colon \RR  \rightarrow \RR$ is an unknown univariate function, $\betas \in \RR^d$ is the parameter of interest, and $\epsilon \in \RR$ is the exogenous random noise such that $\EE(\epsilon) =0$. In addition, we assume that the entries of  $X $ are i.i.d. random variables with density $p_0$ and that $\betas $ is $s^*$-sparse, that is, $\betas$ contains only $s^*$ nonzero entries such that $s^* \ll n \ll d$. Moreover, since the norm of $ \betas $ can be absorbed in $f$, we further let  $\| \betas\|_2 = 1$ for identifiability. Finally, we assume $f$ and $X$ are such that $\EE[ f'( \la X, \betas \ra ) ]  \neq 0$.
\end{definition}

Note that the SIM depends only on covariate only via inner products. Hence it is natural to generalize it to the case of matrix and tensor valued covariates. To enable estimation in a high-dimensional setting, we enforce low-rank constraints that we describe below. 
\begin{definition}[Matrix SIM with First-order Links]
For the low-rank case SIM, we assume that $\beta ^* \in \RR^{d_1 \times d_2}$ has rank $r ^* \ll \min \{ d_1, d_2\}$. In this scenario,  $X \in \RR^{d_1 \times d_2}$ and the inner product  in \eqref{eq:simmodel} is $\la  X , \beta^* \ra = \trace(X^{\top}\beta^*).$ For model identifiability, we further  assume that $\| \beta ^* \|_{\fro} = 1$, similar to the sparse case. Finally, we assume $f$ and $X$ are such that $\EE[ f'( \la X, \betas \ra ) ]  \neq 0$.
\end{definition}
Before we lay out the first-order low-rank tensor single index model, we first introduce additional notation for tensors. Denote by $ u  \otimes v   \otimes s \otimes t \in\mathbb{R}^{d ^ \otimes 4 }  $ a rank-one tensor. The minimum value of $r$ such that the tensor $Z$ could be written as a summation of $r$ rank-one tensors, i.e.,
$
Z = \sum_{j= 1}^r  u_j \otimes v_j \otimes s_j \otimes  t_j,
$
is called as the CP-rank of the tensor, denoted by $\rank_{CP}(Z) = r$.  We now describe the low-rank tensor model that we consider in this work.

\begin{definition}[Tensor SIM with First-order Links]\label{def:tensorcpranksim}
For the low-rank Tensor SIM model, we assume that $\beta ^* \in \RR^{d^{\otimes 4}}$ and has CP-rank, $\rank_{CP}(\beta^*)=r^*$. In this scenario, $X \in \RR^{d^{\otimes 4}}$ and the inner product  in \eqref{eq:simmodel} is understood as defined in \eqref{eq:tensor_inner}. For model identifiability, we further  assume that $\| \beta ^* \|_{\fro} = 1$, similar to the matrix case. Finally, we assume $f$ and $X$ are such that $\EE[ f'( \la X, \betas \ra ) ]  \neq 0$.
\end{definition}

\subsection{Second-order Link Functions}
In the above models, it is  crucial that  $\EE[ f'( \la X, \betas \ra ) ]  \neq 0$, for it to work.  Such a restriction prevents it from being applicable to some widely used cases of SIM, for example, phase retrieval where $f $ is the quadratic function. 
This limitation of the first order Stein's identity, motivates us to examine the second order Stein's identity which is summarized below.
\begin{proposition}[Second-order Stein's Identity~\citep{janzamin2014score}]\label{prop:2stein}
Assume  that the density of $X$ is twice differentiable. In addition, we define the second-order score function  $T\colon \RR^d \rightarrow \RR^{d\times d}$ as $$T(x) = \nabla ^2 p(x) / p(x).$$ Then,  for any  twice differentiable function
$g: \mathbb{R}^d \rightarrow \mathbb{R}$ such that  $\EE [ \nabla ^2 g(X) ]$ exists,  we have
\#\label{eq:second_stein_result}
 \EE \bigl [ g(X)  \cdot T(X) \bigr ]   =  \EE \bigl [ \nabla ^2 g(X)  \bigr ]  . \#
\end{proposition} 

Back to the phase retrieval example, when   $X \sim N(0, I_d)$, the second order score function now becomes  $T(x) = x x^\top - I_d,$ $\forall x \in \RR^d$. Setting $g(x) =  \la x , \betas \ra ^2$ in \eqref{eq:second_stein_result},  we have
\#\label{eq:phase_stein}
\EE [ g(X) \cdot T(X) ] &=\EE[ g(X) \cdot ( X X ^\top - I)] \\
& = \EE [ \la X , \betas \ra ^2 \cdot ( X X ^\top - I)] = 2 \betas {\betas } ^\top . \nonumber
\#
Thus for phase retrieval, one could extract $\pm \betas$ based on second order Stein's identity even in the situation where the first order Stein's identity fails. Indeed, \eqref{eq:phase_stein} used in~\cite{candes2015phasewf} \emph{implicitly} to provide a spectral initialization for the Wirtinger flow algorithm in the case of Gaussian phase retrieval. Here, we provided an alternative justification based on Stein's identity, for why such an initializer works. Motivated by the this observation, we propose to use the second order Stein's identity to  estimate the parametric component of SIMs and MIMs with a class of unknown link functions with non-Gaussian covariates. The precise statistical models that we consider are defined as follows.
\begin{definition}[Vector SIM with Second-order Links] \label{def:simmodel}
Under this model, we assume that the response variable $Y \in \RR$ and the covariate $X\in \RR^d$ are linked via
\#\label{eq:simmodel2}
Y = f  (   \la X,  \betas \ra ) + \epsilon,
\#
where $f\colon \RR  \rightarrow \RR$ is an unknown univariate function, $\betas \in \RR^d$ is the parameter of interest, and $\epsilon \in \RR$ is the exogenous random noise such that $\EE(\epsilon) =0$. In addition, we assume that the entries of  $X $ are i.i.d. random variables with density $p_0$ and that $\betas $ is $s^*$-sparse, that is, $\betas$ contains only $s^*$ nonzero entries. Moreover, since the norm of $ \betas $ can be absorbed in $f$, we further let  $\| \betas\|_2 = 1$ for identifiability. Finally, we assume $f$ and $X$ are such that $ \EE[ f''( \la X, \betas \ra )] >  0$.
\end{definition}
Note that in the definition of the SIMs, we require that $ \EE[ f''( \la X, \betas \ra )]  $ positive. Since if $ \EE[ f''( \la X, \betas \ra )]  $ is negative ,  we could always replace $f$ by $-f$ by flipping the sign of $Y$, we  essentially assume that   $ \EE[ f''( \la X, \betas \ra )]  $ is nonzero.  Intuitively, such restriction on $f$  implies that the second order moments contains the information of $\beta^*$, thus we call such a function the second order link. Similar to the first-order case, one could define matrix and tensor versions of the second-order SIMs but we do not concentrate on such models in this work.  Thus far, we considered SIMs. We now define a class of MIMs with second order links. For MIMs the notion of first order link functions is naturally not sufficient to estimate the projector onto the subspace.  
\begin{definition}[MIM with Second-order Links] \label{def:mimmodel}
Under this model, we assume that the response variable $Y \in \RR$ and the covariate $X\in \RR^d$ are linked via
\#\label{eq:mimmodel}
Y = f \left( \langle X, \beta^*_1\rangle, \ldots, \langle X, \beta^*_k\rangle\right) + \epsilon,
\#
where $f\colon \RR^k  \rightarrow \RR$ is an unknown function, $\{\beta^*_\ell\}_{\ell \in [k] }\subseteq
\RR^d$ are the parameters of interest, and $\epsilon \in \RR$ is the exogenous random noise such that $\EE(\epsilon) =0$. In addition, we assume that the entries of  $X $ are i.i.d. random variables with density $p_0$ and that $\{\beta^*_\ell\}_{\ell \in [k] }$ span a $k$-dimensional subspace of $\RR^d$. Moreover, we denote  $B ^* = (\beta^*_1 \ldots \beta^*_k) \in \RR^{d\times k} $. Then the model in \eqref{eq:mimmodel} can be written as $Y = f( X B^* ) + \epsilon$.  By the QR-factorization, we can write $B^*$  as $Q^* R^*$, where $Q^* \in \RR^{d\times k}$ is an orthonormal matrix and $R^* \in \RR^{k \times k }$ is invertible. Since $f$ is unknown, $R^* $ can be absorbed into the link function. Thus, we assume that $B^*$ is orthonormal for identifiability. Furthermore, we further assume that $B^*$ is $s^*$-row sparse, that is, $B^*$ contains only $s^*$ nonzero rows. We note that such a definition of sparsity for $B^*$ does not depends on the choice of coordinate system. Finally,  we assume $f$ and $X$ are such that $ \lambda_{\min} \left(\EE[ \nabla ^2 f(  XB^*  )] \right)>  0$.
\end{definition}

The assumption $\EE[ \nabla ^2 f( X B^* )] $ is positive definite, in MIM, is a multivariate generalization of the condition that $ \EE[ f''( \la X, \betas \ra )] >  0$ in SIM. It essentially guarantees that estimation of the projector onto the subspace spanned by the $k$ components is well-defined. We now introduce our estimators and provide theoretical results that are near-optimal in several settings. 

\section{Theoretical Results for Index Models with First-order Links}\label{sec:firstordertheory}
Recall that in the single index models introduced in \S \ref{sec:firstorder}, $X$ in \eqref{eq:simmodel} has i.i.d. entries with density $p_0$.  To unify the vector, matrix and tensor settings, we identify $X$ with $\vec(X) \in  \RR^{d }$ where $d = d_1\cdot d_2 \cdot d_3 \cdot d_4$. In this case, $X$ has density  $p = p_0^{\otimes d}$ and the corresponding  score function  $S \colon \RR^{d } \rightarrow \RR  ^{d  }$ is given by
\#\label{eq:def_score}
S (x) = - \nabla \log p(x)=  - \nabla p (x) / p  (x)  =  s_0\circ(x ),
\#  
where  the univariate function $s_0 =- p_0'  /  p_{0} $ is applied to each entry of $x$.
Thus   $S(X) $ has i.i.d. entries. In addition, by Proposition \ref{prop:1stein},  we have $\EE [S(X) ] = 0$ by setting $g$ to be a constant function. Moreover, in the context of SIMs specified in \eqref{eq:simmodel}, we have 
\$
\EE [ Y \cdot   S(X)    ] &= \EE \bigl [ f( \la X , \beta^*\ra) \cdot S(X) \bigr] \\ &= \EE   [ f' ( \la X ,  \beta ^* \ra ) ] \cdot  \beta ^*, \nonumber
\$
as long as the density and the link function satisfy the conditions stated in Proposition~\ref{prop:1stein}. This implies that optimization problem 
\#\label{eq:opt_population}
\minimize _{\beta \in \RR^d }  \bigl\{ \la \beta , \beta \ra - 2 \EE [ Y \cdot   \la S(X) , \beta \ra    ]  \bigr\} 
\#
has solution $\beta =  \mu \cdot \beta ^* $, where $ \mu =\EE   [ f' ( \la X ,  \beta ^* \ra ) ] $. Hence the above program could be used to obtain the unknown $\beta^*$ as long as $\mu \neq 0$. Before we proceed to describe the sample version of the above program, we make the following brief remark. The requirement $\mu \neq 0$ rules out in particular the use of our approach for non-Gaussian phase retrieval (where $f(u) = u^2$) as in that case we have $\mu =0$ when $X$ is centered. But we emphasize that the same holds true in the Gaussian and elliptical setting as well, as noted in~\cite{plan2015generalized} and \cite{goldstein2016structured}. Their methods also fail  to recover the true $\beta^*$ when the SIM model corresponds to phase retrieval. We refer the reader to \S\ref{sec:secondordertheory} for overcoming this limitation using second-order Stein's identity.

We use a sample version of the above program as an estimator for the unknown $\beta^*$. In order to deal with the high-dimensional setting, we consider a regularized version of the above formulation. More specifically, we use the $\ell_1$-norm and nuclear norm regularization in the vector and matrix/tensor settings respectively. However, a major difficulty in the sample setting for this procedure is that $\EE[ Y \cdot S(X) ]$ and its empirical counterpart may not be close enough due to  a lack of concentration. Recall our discussion from \S\ref{sec:intro} that even if the random variable $X$ is light-tailed, its score-function $S(x)$ might be arbitrarily heavy-tailed.  Furthermore,   bounded-fourth moment assumption on the noise, $Y$ too can be heavy-tailed.  Thus the naive method of using the sample version of  \eqref{eq:opt_population} to estimate $\beta^*$ leads to sub-optimal statistical rates of convergence.

To improve concentration and obtain optimal rates of convergence, we replace $Y \cdot S(X)$ with a transformed random variable $\cT(Y, X)$, which will be defined precisely later  for the sparse and low-rank cases. In particular, $\cT(Y, X)$ is a carefully truncated version of $Y \cdot S(X)$, introduced and analyzed in \cite{catoni2012challenging, fan2016robust} for related problems, that enables us to obtain well-concentrated estimators. Thus our final estimator $\hat \beta$ is defined as the solution to  the following regularized  optimization  problem
\#\label{eq:estimator}
\minimize _{ \beta \in \RR^d} ~L(\beta)   + \lambda \cdot R(\beta),
\#
where 
\# \label{eq:define_a_loss}
L(\beta) = \la \beta , \beta \ra   - \frac{2}{n} \sum_{i=1}^n  \la \cT(Y_i, X_i) ,  \beta\bigr \ra, 
\# and $\lambda>0$ is the regularization parameter which will be specified later and $R(\cdot)$ is the $\ell_1$-norm in the vector case and the nuclear norm in the matrix/tensor case. 
We now introduce our main moment assumption for first-order SIM. This assumption is made apart from the assumptions made on the noise and the link function. Recall that each entry of the score function defined in \eqref{eq:def_score} is equal to $s_0(u) = - p_0' (u) / p_0(u)$. We first state the assumption and make a few remarks about it.

\begin{assumption}[Moment Assumptions]\label{assume:moments_sim}
 There exists an absolute constant $M >0$ such that $\EE (Y^4 )  \leq M$ and $\EE _{p_0} [s^4_0(U) ]  \leq M$, where random variable $U \in \RR$ has   density $p_0$.
\end{assumption}
Consider the assumption $\EE (Y^4 )  \leq M$. By Cauchy-Schwarz inequality we have $\EE (Y^4) \leq  4 \EE (\epsilon^4) + 4  \EE [ f^4 (\la X, \beta^* )]$. Note that we assume  $\epsilon $ to be centered, independent of $X$ and has bounded fourth moment (see \S\ref{sec:model}). If the covariate $X$ has bounded fourth moment along the direction of true parameter, since $f(\cdot)$ is continuously differentiable, $f(\la X, \beta ^*\ra)$  has bounded fourth moment as well if  $f(\cdot)$ is defined on a compact subset of $\RR$.  . Hence the condition $\EE (Y^4 )  \leq M$ is relatively easy to satisfy and significantly milder than assuming that $Y$ is bounded or has lighter tails. Furthermore, $\EE _{p_0} [s^4_0(U) ]  \leq M$ is relatively mild and satisfied by a wide class of random variables. Specifically random variables that are non-symmetric and non-Gaussian satisfy this property thereby allowing our approach to work with covariates not previously possible. We believe it is highly non-trivial to weaken this condition without losing significantly in the rates of convergence that we discuss below.


\subsection{Sparse Vector SIM}
Under the above assumptions, we first state our theorem on the sparse SIM. As discussed above,  $Y \cdot S(X)$ can by heavy-tailed and hence we apply truncation to achieve concentration. Denote the  $j$-th entry of the score function $S$   in \eqref{eq:def_score} as $S_j \colon \RR^d \rightarrow \RR$, $j \in[d]$.  We define the truncated response and score function as
\#\label{eq:thresholding}
\tilde Y    &= \sign(Y ) \cdot ( | Y| \wedge \tau),\\ S_j   (x) &= \sign [  S_j (x)  ]  \cdot \bigl [ | S_j (x) | \wedge \tau \bigr ], \nonumber
\#
where $\tau>0$ is a  predetermined threshold value. We define $\tilde Y_i$ similarly for all $Y_i$, $i\in [n]$. Then we define  the  estimator $\hat \beta$ as the solution to  the optimization problem in \eqref{eq:estimator} with $\cT(Y_i, X_i) = \tilde Y_i \cdot   \tilde S(X_i) $ and $R(\beta ) = \| \beta \|_1$. Here we apply elementwise truncation in $\cT$ to ensure the sample average of $\cT $ converges to $\EE [ Y \cdot S(X)]$ in the $\ell_{\infty}$-norm for an appropriately chosen $\tau$. Note that the $\ell_{\infty}$-norm is the dual norm of the $\ell_1$-norm. Such a convergence requirement in the dual norm is standard in the analysis of regularized $M$-estimators \citep{negahban2012unified} to achieve optimal rates. The following theorem characterizes the convergence rates of $\hat \beta$.

\begin{theorem} [Signal Recovery for Sparse Vector SIM] \label{thm:sparse} For the sparse SIM defined in \S\ref{sec:model}, we assume that $\beta ^* \in \RR^d$ has $s^*$ nonzero entries.   Under Assumption \ref{assume:moments_sim}, we let $$\tau  = 2 (M \cdot \log d / n)^{1/4}$$ in \eqref{eq:thresholding} and  set the regularization parameter  $\lambda   $ in \eqref{eq:estimator}  as $$\lambda=C \sqrt{M \cdot \log d/n},$$ where $C>0$ is an absolute constant. Then with probability at least $1- d^{-2}$,
 the $\ell_1$-regularized estimator  $\hat\beta $ defined in  \eqref{eq:estimator} satisfies
\$
\| \hat \beta - \mu \beta^* \|_2 \leq  \sqrt{s^*} \cdot \lambda, ~~\| \hat \beta  - \mu \beta^* \|_1\leq  4 s^*\cdot \lambda.
\$
\end{theorem}
From this theorem,  the  $\ell_1$- and $\ell_2$-convergence rates of $\hat \beta$ are $ \| \hat \beta - \mu \beta^* \|_1 = \cO(s^* \sqrt{\log d/n})$ and  $ \| \hat \beta - \mu \beta^* \|_2 = \cO( \sqrt{s^* \log d/n})$, respectively. These rates match  the convergence rates of sparse generalized linear models \citep{loh2015regularized} and sparse single index models with Gaussian and symmetric elliptical covariates \citep{plan2015generalized, goldstein2016structured} which are known to be minimax-optimal for this problem via matching lower bounds.

\subsection{Low-rank Matrix SIM}  \label{sec:lowrank}

We next state our theorem for the low-rank Matrix SIM. In this case, we apply the nuclear norm regularization to promote low-rankness. Note that by definition, $\cT$ is matrix-valued. Since the dual norm of the nuclear norm is the operator norm, we need the sample average of $\cT$ to converge to $\EE [ Y \cdot S(X) ]$ in the operator norm rapidly to achieve optimal rates of convergence. To achieve such a goal, we leverage the truncation argument from \cite{catoni2012challenging, minsker2016sub, fan2016robust} to construct  $\cT(Y, X)$.

Let $\phi \colon  \RR \rightarrow \RR $ be a non-decreasing function such that
$$
-\log(1  - x + x^2 /2 )\leq \phi(x) \leq  \log(1 + x + x^2 /2), ~~\forall x \in \RR.
$$
Based on $\phi$, we define a linear mapping $\psi \colon \RR^{d_1 \times d_2}  \rightarrow \RR^{d_1 \times d_2}$ as follows.  For any $A \in \RR^{d_1 \times d_2}$, let  \$  \tilde A = \begin{bmatrix} 0 & A \\
A^{\top} & 0
\end{bmatrix}   \$  and let $\Upsilon \Lambda \Upsilon ^\top$ be the eigenvalue composition of $\tilde A$.  In addition, let $B =  \Upsilon \bigl [  \psi\circ(\Lambda) \bigr ] \Upsilon ^\top$, where $\psi$ is applied elementwisely on $\Lambda$. Then  we  write $B$ in block from as
\$
B = \begin{bmatrix} B_{11} & B_{12} \\
B_{21} & B_{22}
\end{bmatrix}
\$
and define $   \psi(A) = B_{12}$.  Finally, we define
$\cT(Y, X) = 1/ \kappa \cdot \psi\bigl [ \kappa \cdot  Y \cdot     S(X ) \bigr ],  $ where $\kappa >0$   will be specified later.  Therefore, our final estimator  $\hat \beta  \in \RR^{d_1 \times d_2}$ is defined as the solution to the optimization problem in \eqref{eq:estimator}  with $R (\beta) = \| \beta \|_{\nuc} $.  We note  here the minimization in    \eqref{eq:estimator} is taken over $\RR^{d_1 \times d_2}$. The following theorem quantifies the convergence rates of the proposed estimator.

\begin{theorem} [Signal Recovery for Low-rank Matrix SIM] \label{thm:lowrank} For the low-rank single index model  defined in \S\ref{sec:model}, we assume that $\rank (\beta ^* ) = r^*$.    Under Assumption \ref{assume:moments_sim}, we let  $$\kappa = 2  \sqrt{n \cdot \log (d_1 + d_2)}  /\sqrt{  (d_1 + d_2 )M}$$ in $\cT(Y, X)$ and set   $\lambda   $ in \eqref{eq:estimator}  as $$\lambda = C \sqrt{  M\cdot ( d_1+d_2)\cdot \log (d_1+ d_2)/n},$$ where $C>0$ is an absolute constant. Then with probability at least $1- (d_1+d_2)^{-2}$,
 the  nuclear norm regularized estimator  $\hat\beta $  satisfies
\$
\| \hat \beta  - \mu \beta^* \|_{\fro} \leq  3 \sqrt{r^*} \cdot \lambda, ~~\| \hat \beta  - \mu \beta^* \|_{\nuc}\leq  12 r^*\cdot \lambda.
\$
\end{theorem}
By this theorem, we have $\| \hat \beta - \mu \beta ^* \|_{\fro} = \cO( \sqrt{r^* (d_1+ d_2) \cdot\log (d_1+ d_2)/ n} ) $ and $\| \hat \beta - \mu \beta ^* \|_{\nuc} = \cO( r^*\cdot \sqrt{ (d_1+ d_2)\cdot \log (d_1+ d_2)/ n} ) $. Note that the rate obtained is minimax-optimal up to a logarithmic factor. Furthermore, it matches the rates for low-rank single index models with Gaussian and symmetric elliptical distributions  up to a logarithmic factor~\cite{ plan2015generalized, goldstein2016structured}.

\subsection{Low-rank Tensor SIM}  \label{sec:lowrankten}
We now state our result for low-rank tensor SIM. The notion of rank of a tensor is more delicate compared to that of a matrix. Several generalizations of the matrix rank exist for the case of tensors. Recall from Definition~\ref{def:tensorcpranksim}, that we assumed that the structure on $\beta^*$ is that it has low CP-rank.  Unfortunately, enforcing such a constraint via a direct tensor nuclear norm relaxation (similar to that of the matrix nuclear norm) is NP-hard~\citep{friedland2014nuclear}. 

One way to overcome such a computational hurdle is to deal with tensors via appropriately matricized forms. In order to enable computable estimators, we specifically leverage the results of~\cite{mu2014square} and define the following square-unfolding of a tensor. Denote by $ \mat \colon \RR^{d ^ \otimes 4  }  \rightarrow \RR^{d^2 \times d^2 }$ the operation of tensor square-unfolding, which maps a fourth-order tensor to a square matrix. More specifically, the entries of $\mat(Z)$ are specified by
$ [ \mat (Z)]_{k_1, k_2} = Z ( j_1, j_2, j_3, j_4),
$
where the indices satisfy the relationship $k_1 = 1 + (j_1 -1) + (j_2-1)\cdot d$ and $k_2 = 1+ (j_3-1) + (j_4-1) \cdot d$. Intuitively, the matrix obtained by the square-unfolding operation is as square as possible, i.e., it is $d^2 \times d^2$ rather than the rectangular $d \times d^3$ or $d^3 \times d$ matrices. It is shown in~\cite{mu2014square} such a square matricization preserves the low CP-rank of the original tensor. Hence one could use the matrix nuclear norm relaxation on the square-unfolded tensor. Furthermore, for the case of tensor SIM as in Definition~\ref{def:tensorcpranksim}, note that we have $  \la X,  \betas \ra = \la\mat(X), \mat(\betas)\ra$. Combining the above observations, the low CP-rank tensor SIM problem could be reduced to that of low-rank matrix SIM problem, where matrix low-rank constraint, via nuclear norm, is enforced on $\mat(\betas)$. Thus, we use the estimator in~\eqref{eq:estimator} with $R(\beta) = \| \mat(\beta)\|_*$ and $\mat(X_i)$ for all $i=1, \ldots, n$ along with the truncation operation $\mathcal{T}$ described in \S\ref{sec:lowrank}. We now have the following theorem for the low-rank tensor SIM.

\begin{theorem} [Signal Recovery for Low-rank Tensor SIM] \label{thm:lowranktensor} For the low-rank tensor single index model  defined in \S\ref{sec:model}, Definition~\ref{def:tensorcpranksim}, we assume that $\rank_{CP} (\beta ^* ) = r^*$.    Under Assumption \ref{assume:moments_sim}, we let  $$\kappa = 2  \sqrt{2 n \cdot \log d}  /\sqrt{  (2d^2)M}$$ in $\cT(Y, X)$ and set   $\lambda$ in \eqref{eq:estimator} as $$\lambda=C \sqrt{ 2 M\cdot ( 2d^2) \cdot \log d /n},$$ where $C>0$ is an absolute constant. Then with probability at least $1- (2d)^{-2}$,
 the  nuclear-norm regularized estimator  $\hat\beta $  satisfies
\$
\| \hat \beta  - \mu \mat(\beta^*) \|_{\fro} \leq  3 \sqrt{r^*} \cdot \lambda
\$
\end{theorem}

We omit the proof of the above theorem as it is follows the exact steps of Theorem~\ref{thm:lowrank} proved in Appendix~\ref{app:lr}. From the above theorem, we see that as long as $n = \Omega (r^* d^2)$, we achieve consistent estimation of $\beta^*$ up to scaling. This improves upon recent results established in~\cite{chen2016non}, that established similar results under restrictive Gaussian covariate assumption and required knowledge of the link functions (i.e., generalized linear models). Furthermore our results significantly generalizes the results of~\cite{mu2014square} that considered only linear link functions.  Finally, although our structure on $\beta^*$ was a low CP-rank structure, the square matricization technique also applies for the case of low Tucker-rank, which is yet another notion of rank for tensors with several applications. It is straightforward to extend our results to this case of low Tucker-rank.

\section{Theoretical Results for Index Models with Second-Order Links}\label{sec:secondordertheory}

We now introduce our estimators and establish their statistical rates of convergence for the case of index models with second-order link functions. Discussions on  optimality of the established rates and connection to sparse PCA problem  is deferred to \S\ref{sec:optspca}. Similar to the first-order case, we focus on the case where $X$ has i.i.d. entries with density $p_0  \colon \RR \rightarrow \RR$. Thus the joint density of $X$ is $p(x) = p_0^{\otimes d} (x) = \prod_{j=1}^d p_0(x_j)$.  We define a univariate  function $s_0 \colon \RR\rightarrow \RR$ by $s_0(u) =  -p_0' (u) / p_0(u)$. Then the first-order score function associated with $p$ is given by $S  (x) = s_0 \circ (x)$. Equivalently,  the $j$-th entry of the first-order score function associated with $p$ is given by  $[ S(x) ]_{j} = s_0(x_j)$. Moreover, the second order score function is
\#\label{eq:second_score}
T (x) & =  S (x)  S^\top(x) - \nabla S(x) \\ &=  S (x)  S^\top(x) - \diag[s_0'\circ (x)]. \nonumber
\#
Before we present our estimator, we introduce  the assumption  on $Y$ and $s_0(\cdot)$. 

\begin{assumption}[Moment Assumptions]\label{assume:moments}
 We  assume that there exists a constant $M$ such that $\EE_{p_0} [ s_0^6(U)] \leq M$ and $ \EE (Y^6) \leq M$. We denote $\sigma_0^2 = \EE_{p_0} [ s_0^2 (U)] = \Var_{p_0} [ s_0(U)]$.
\end{assumption}

The assumption that $\EE_{p_0} [ s_0^6(U)] \leq M$ allows wide family of distributions of   including Gaussian and more heavy-tailed random variables. Furthermore, we do not  require the covariate  to be  elliptically symmetric as is commonly seen in prior work, which enables our method to be  applicable for skewed covariates. As for the assumption that   $ \EE (Y^6) \leq M$, in the case of SIMs,  we have $ \EE (Y^6) \leq C\left (\EE (\epsilon^6) +\EE [f ^6(\la X,  \betas \ra ) ] \right)$. Thus this assumption is satisfied  as long as both $\epsilon$  and $f( \la X , \betas \ra ) $ have bounded sixth moments.   This is a significantly milder assumption which allows for heavy-tailed response as opposed to bounded or light-tailed response.   
\subsection{Sparse Vector SIM}
Now we are ready to describe our estimator for the sparse SIMs in Definition \ref{def:simmodel}. 
Note that by Proposition \ref{prop:2stein} we have
\#\label{eq:apply2stein}
\EE \bigl [ Y \cdot T (X)\bigr  ] =  C_0 \cdot \betas {\betas}^\top,
\#
where $C_0  = 2 \EE[ f''( \la X, \betas \ra )] >0 $ as per Definition \ref{def:simmodel}. Therefore, one way to estimator  $\betas$ is to obtain  the leading eigenvector of $\EE [ Y \cdot T (X) ]$ from the samples. Since $\betas$ is sparse, we formulate our estimator as a sparsity constrained semi-definite program:
 \#\label{eq:sparsePCApop}
&\textrm{maximize}~~ \la W, \Sigma^*\ra - \lambda \| W \|_{1 } \notag \\
&\text{subject to}~~ 0 \preceq W \preceq I_d, ~~\trace (W ) = 1.
\#
where $\Sigma^* = \EE [ Y \cdot T (X)   ]$. Note that both the score $T (X)$ and the response variable $Y$ can be heavy-tailed. In order to obtain near-optimal estimates in the sample setting, we apply truncation to handle the heavy-tails. Specifically, for a positive parameter   $\tau \in \RR $,  we define the truncated random variables  by
\# \label{eq:truncation}
\tilde Y_i &= \sign (Y_i) \cdot \min \{ | Y_i| , \tau   \} \\
 \tilde T _{jk}(X_i)  & = \sign  \{  T_{jk}  (X_i)  \} \cdot \min \bigl \{  | T _{jk} (X_i)  |  , \tau ^2 \bigr \}. \nonumber
\#
Then we define an robust estimator of $\Sigma^*$ as
\#\label{eq:estimate_cov}
\tilde \Sigma = \frac{1}{n} \sum _{i=1}^n \tilde Y_i \cdot \tilde T (X_i) .
\#
Given $\tilde \Sigma$, let $\hat W$ be the solution of
the following convex optimization problem  
\#\label{eq:sparsePCA1}
&\textrm{maximize}~~ \la W, \tilde \Sigma \ra  -  \lambda \| W \|_{1 } \notag \\
&\text{subject to}~~ 0 \preceq W \preceq I_d, ~~\trace (W ) = 1.
\#
Here $\lambda$ is a regularization parameter to be specified later. The  final estimator is defined as the leading eigenvector of $\hat W$. The following theorem quantifies the statistical rates of convergence of the proposed estimator.

\begin{theorem}[Signal Recovery for Sparse SIM]\label{thm:pr1}
Let $\hat W $ be the solution of the optimization in \eqref{eq:sparsePCA1} and let $\hat \beta  $ be the leading eigenvector of $\hat W$. We set the regularization parameter $\lambda$ in \eqref{eq:sparsePCA1} as $\lambda = 10  \sqrt{M   \log d / n}$ and set $\tau = (1.5 M n /\log d)^{1/6}$ in \eqref{eq:truncation}. Under Assumption \ref{assume:moments}, we have $\min_{t \in \{+1,-1 \}}\| \hat \beta  - t \betas \|_2 \leq 4 \sqrt{2} s^* \lambda $  with probability at least $1 - d^{-2}$.
\end{theorem}
By this theorem, the $\ell_2$-error of the proposed estimator is $\cO(s^* \sqrt{ \log d / n})$, which 
implies that consistent estimation requires $n = \Omega( {s^*}^2 \log d)$ samples. 

\subsection{Subspace-Sparse MIM}
Now we introduce the estimator for $B^*$ of the sparse MIM in Definition \ref{def:mimmodel}. Proposition \ref{prop:2stein} implies that $\EE[ Y \cdot T(X)]  = B^* D_0 B^* $, where $D_0 = \EE [ \nabla ^2 f(X B^*)]$ is positive definite. Similar to \eqref{eq:sparsePCA1},  we recover the column space of  $ B^* $ by solving
\#\label{eq:sparsePCA2}
&\textrm{maximize}~~ \la W, \tilde \Sigma \ra  -  \lambda \| W \|_{1 }, \notag \\
&\text{subject to}~~ 0 \preceq W \preceq I_d, ~~\trace (W ) = k.
\#
where $\tilde \Sigma $ is defined in \eqref{eq:estimate_cov}, $\lambda >0 $ is a regularization parameter and $k$ is the number of indices which is assumed to be known. Let $\hat W$ be the solution of  \eqref{eq:sparsePCA2}, the final estimator is the top $k$ eigenvectors of $\hat W$. For the above estimator, we have the following theorem quantifying the statistical rate of convergence. Let $\rho_0 = \lambda_{\textsc{min}}\left(\EE [ \nabla ^2 f ( X B^*)] \right)$. 

\begin{theorem}[Signal Recovery for Sparse MIM]\label{thm:pr2}
Let $\hat W $ be the solution of the optimization problem in \eqref{eq:sparsePCA2} and let $\hat B  $ be the matrix of top-$k$ eigenvectors of $\hat W$. We set the regularization parameter in \eqref{eq:sparsePCA1} as $\lambda = 10 \sqrt{M   \log d / n}$ and let  the truncation parameter in \eqref{eq:truncation} be $\tau = (1.5 M n /\log d)^{1/6}$ . Under Assumption \ref{assume:moments}, with probability at least $1 - d^{-2}$, we have
\$
\inf_{O\in \mathbb{O}_{k} }\|\hat B  - B^* O \|_2 \leq 4 \sqrt{2} / \rho_0 \cdot s^* \lambda .
\$
\end{theorem}
Minimax lower bounds for subspace estimation for MIM was established recently in~\cite{lin2017optimality}. For a fixed $k$, the above theorem is near-optimal from a minimax point of view. That is, the difference between the optimal rate and the above theorem is a factor of $\sqrt{s}$. We discuss more about this gap in Section~\ref{sec:optspca}. The proofs of Theorem~\ref{thm:pr1} and Theorem~\ref{thm:pr2} are in the supplementary material.

\begin{remark}
Recall that our discussion in \S\ref{sec:firstordertheory} and \S\ref{sec:secondordertheory} was under the assumption that the entries in $X$ are i.i.d. This could be relaxed to the case of weak dependence between the covariates without   any significant loss in the statistical rates we present in the theorems above. We do not focus on such an extension in this paper as we wanted to clearly convey the main message of the paper in a simpler setting.
\end{remark} 

\subsection{Optimality and Relation to Sparse PCA}\label{sec:optspca}

Now we discuss  the optimality of the results presented in \S\ref{sec:secondordertheory}. Throughout the discussion we assume that $k$ is fixed and does not increase with $n$. Note that the estimator for SIM in  \eqref{eq:sparsePCA1} and MIM in  \eqref{eq:sparsePCA2} are closely related to the semidefinite program based estimator for Sparse PCA problem~\citep{vu2013fantope}.    Let $X \in \mathbb{R}^d$ be a random vector such that $\EE(X) =0$ and covariance matrix $\Sigma= \EE (XX^\top)$ which is symmetric and positive definite. The problem of sparse PCA is to estimate projector onto the subspace spanned by top $k$ eigenvectors, $\{v^*_\ell \}_{\ell \in [k]}$ of  $\Sigma$ under the subspace sparsity assumption as discussed in Definition~\ref{def:mimmodel}. An estimator based on semidefinite programing with sparsity constraints was analyzed in~\cite{vu2013fantope, wang2016statistical}, which   is based on solving the following program 
\#\label{eq:sparsePC}
&\textrm{maximize}~~ \la W, \hat \Sigma \ra  -  \lambda \| W \|_{1 } \notag \\
&\text{subject to}~~ 0 \preceq W \preceq I_d, ~~\trace (W ) = k.
\#
Here $\hat \Sigma=n^{-1} \sum_{i=1}^n X_i X_i^\top$ is the sample covariance matrix given $n$ i.i.d copies $\{ X_i\}_{i=1}^n$ of $X$. Note that the main difference between the SIM estimator and the sparse PCA estimator is the use of $\tilde \Sigma$ in place of $\hat \Sigma$. It is known that sparse PCA problem exhibits interesting statistical-computational tradeoff~\citep{krauthgamer2015semidefinite,wang2016statistical} which naturally appears in the context of SIM as well. Indeed while the minimax optimal statistical rate for sparse PCA is $\cO(\sqrt{s^* \log d/n})$, the SDP estimator achieves $\cO(s^*\sqrt{ \log d/n})$ under the assumption that $X$ is light-tailed.
It is also known that when $n = \Omega( 	 {s^*}^2 \log d)$, one can obtain the optimal statistical rate of $\cO(\sqrt{s^* \log d/n})$ either by nonconvex methods~\citep{wang2014tighten}, or refinements to the output of the SDP estimator~\citep{wang2016statistical}. However their results rely on the sharp concentration of $\hat \Sigma$ to $\Sigma$ in the restricted operator norm:
\# \label{eq:rsnorm}
\| \hat \Sigma - \Sigma ^* \|_{op, s^*}  &=\sup \bigl \{ w ^\top ( \hat\Sigma - \Sigma ) w \colon \| w \|_2 = 1, \| w \|_0 \leq s^* \bigr \}  \nonumber \\ & = \cO(\sqrt{s^* \log d/n}). 
\#
When $X$ has heavy-tailed entries, for example bounded fourth moment assumptions, its highly unlikely that,  \eqref{eq:rsnorm} holds. Indeed the results in~\cite{wang2016statistical} and~\cite{wang2014tighten} are applicable only to the case of Gaussian or light-tailed $X$.  
\subsubsection{Heavy-tailed Sparse PCA}
Recall that our estimators utilize a data-driven truncation argument to handle heavy-tailed distributions. Owing to the close relationship between our SIM/MIM estimators and the sparse PCA estimator, it is natural to ask whether such a truncation argument could lead to sparse PCA estimators for heavy tailed $X$. Below we show that it is indeed possible to obtain a near-optimal estimator for sparse PCA with heavy-tailed data based on the truncation argument. For a vector $v \in \mathbb{R}^d$, let $\vartheta(v)$ be a truncation operation that operators entry-wise as $\vartheta_j(v) = \sign[v_j]\cdot \min \left\{|v_j| , \tau\right\}$ for $j= 1, \ldots d$. Then, our estimator is defined as follows.
\#\label{eq:sparsePCheavy}
&\textrm{maximize}~~ \la W, \overbar \Sigma \ra -  \lambda \| W \|_{1 } \notag \\
&\text{subject to}~~ 0 \preceq W \preceq I_d, ~~\trace (W ) = k.
\#
where  $\overbar \Sigma = n^{-1} \sum_{i=1}^n \overbar{X}_i \overbar{X}_i^\top$ and $\overbar{X}_i =\vartheta(X_i)$, for $i=1,\ldots n$. For the above estimator, we have the following theorem under the assumption that $X$ has heavy-tailed marginals. Let $V ^* = (v^*_1 \ldots v^*_k) \in \RR^{d\times k} $ and assume that $\rho_0 = \lambda_k(\Sigma) - \lambda_{k+1}(\Sigma) > 0$.

\begin{theorem}\label{thm:htspca}
Let $\hat W $ be the solution of the optimization in \eqref{eq:sparsePCheavy} and let $\hat V  $ be matrix of top-$k$ eigenvectorssim of $\hat W$. We set the regularization parameter in \eqref{eq:sparsePCheavy} as $\lambda = C_1 \sqrt{M   \log d / n}$ and set the truncation parameter by  $\tau = (C_2 M n /\log d)^{1/4}$, where $C_1$ and $C_2$ are   some positive constants. Furthermore, assume that $V^*$ contains only $s^*$ nonzero rows and that $X$ satisfies   $\EE| X_j|^4 \leq M$ and $\EE |X_i\cdot X_j|^2 \leq M$. Then, with probability at least $1 - d^{-2}$, we have
\$
\inf_{O\in \mathbb{O}_{k} }\|\hat V  - V^* O \|_2 \leq 4 \sqrt{2} / \rho_0 \cdot s^* \lambda .
\$
\end{theorem}
The proof of the above theorem is similar to that of Theorem~\ref{thm:pr2} and hence we omit it. The above theorem shows that with  elementwise truncation,  as long as $X$ satisfies a bounded fourth moment condition, the SDP estimator for sparse PCA achieves the near-optimal statistical rate of $\cO(s^*\sqrt{ \log d/n})$. We end this section with the following questions based on the above discussions:
\begin{enumerate}
\item Can we obtain optimal statistical rate for sparse PCA problem ($\cO(\sqrt{s^*\log d/n})$) when $X$ has only bounded fourth moment in the high sample size regime $n =\Omega( {s^*}^2 \log d)$?
\item Can we obtain optimal statistical rate ($\cO(\sqrt{s^*\log d/n})$) when $n =\Omega( {s^*}^2 \log d)$ and when $f,X$ and $Y$ satisfies the heavy-tail condition in Assumption~\ref{assume:moments} for the MIM problem?
\end{enumerate}
The answer to both questions lie in constructing truncation based estimators that concentrate sharply in restricted operator norm as in \eqref{eq:rsnorm} or more realistically exhibit one-sided concentration bounds (see~\cite{mendelson2015learning} and~\cite{oliveira2013lower} for related results and discussion). Obtaining such an estimator seems to be  challenging   for heavy-tailed sparse PCA and it it not immediately clear if it is possible. We plan to report our findings for the above problem in the near future.

\section{Numerical Experiments}
We now provide simulation experiments for the case of first-order and second-order SIMs. For the first-order SIM, we concentrate on the sparse vector and low-rank matrix model. Note that our tensor estimator is similar to the low-rank matrix estimator. Furthermore, for the second-order case, we concentrate on the problem of robust sparse phase retrieval. 

\textbf{First-order SIM}: We let $\epsilon \sim N(0,1)$ and  set the link function   in \eqref{eq:simmodel} as one of $f_1(u) = 3 u + 10 \sin(u) $ and  $f _2(u) = \sqrt{2} u + 4 \exp(-2 u^2)$, which are  plotted in Figure \ref{fig:link}. We set $p_0$ to be one  of (i) Gamma distribution with shape parameter  $5$ and scale parameter  $1$, (ii) Student's t-distribution with $5$ degrees of freedom, and (iii) Rayleigh distribution with scale parameter    $2$. To measure the estimation accuracy, we use the cosine distance $ \cos \theta(\hat \beta , \beta^* )  = 1- \| \hat \beta \| _{\bullet}^{-1}  | \la \hat \beta, \beta^* \ra |$, where $\bullet$ stands for the Euclidean norm in the vector case and the Frobenius norm when $\beta^*$ is a  matrix. Here we report the cosine distance rather than  $\| \hat \beta - \mu \beta^* \|_{\bullet}$ to compare the performances for  $X$ having different distributions, where $ \mu $ may have different values.

For the vector case, we fix $d = 2000$, $s^* = 5$ and vary $n$. The support of $\beta^*$  is chosen  uniformly random among all subsets of $\{1, \ldots, d\}$. For each $j \in \supp(\beta^*)$, we set $\beta^*_j = 1/\sqrt{s^*} \cdot \gamma_j$, where each $\gamma_j$ is an i.i.d. Rademacher random variable.   In addition,  the   regularization parameter $\lambda$ is set to $4 \sqrt{ \log d/n}$.  We plot the cosine distance against the signal strength $\sqrt{s ^* \log d/n}$ in Figure \ref{fig:simulation_sparse}-(a) and (b) for $f_1$ and $f_2$ respectively,  based on $200$ independent trials for each $n$. As shown in this figure,  the estimation error grows sub-linearly as a function of the signal strength.

As for the matrix case, we fix $d_1 = d_2  =    20$, $r^* = 3$ and let $n$ vary. The signal parameter $\beta^*$ is equal to  $U S V^\top$, where $U, V \in \RR^{d\times d}$ are random orthogonal matrices and $S$ is a diagonal matrix with $r^*$ nonzero entries. Moreover, we set the nonzero diagonal entries of $S$ as $1/ \sqrt{r^*}$, which implies $\| \beta^* \|_{\fro } = 1$. We set   the regularization parameter as $\lambda = 2   \sqrt{ (d_1 + d_2)  \log (d_1 + d_2) / n}$.  Furthermore, we  use the proximal gradient descent algorithm (with the learning rate fixed to $0.05$) to solve  the nuclear norm regularization  problem in \eqref{eq:estimator}. To present the result, we plot the cosine distant against  the signal strength $\sqrt{r ^* (d_1 + d_2)  \log (d_1 + d_2) / n} $ in Figure~\ref{fig:simulation_lowr}  based on  $200$ independent trials for both $f_1$ and $f_2$.  As shown in this figure, the error is bounded by a linear function of the signal strength, which corroborates  Theorem \ref{thm:lowrank}.

\begin{figure}[htb]
	\centering
	\begin{tabular}{cc}
		\hskip-27pt\includegraphics[width=0.22\textwidth]{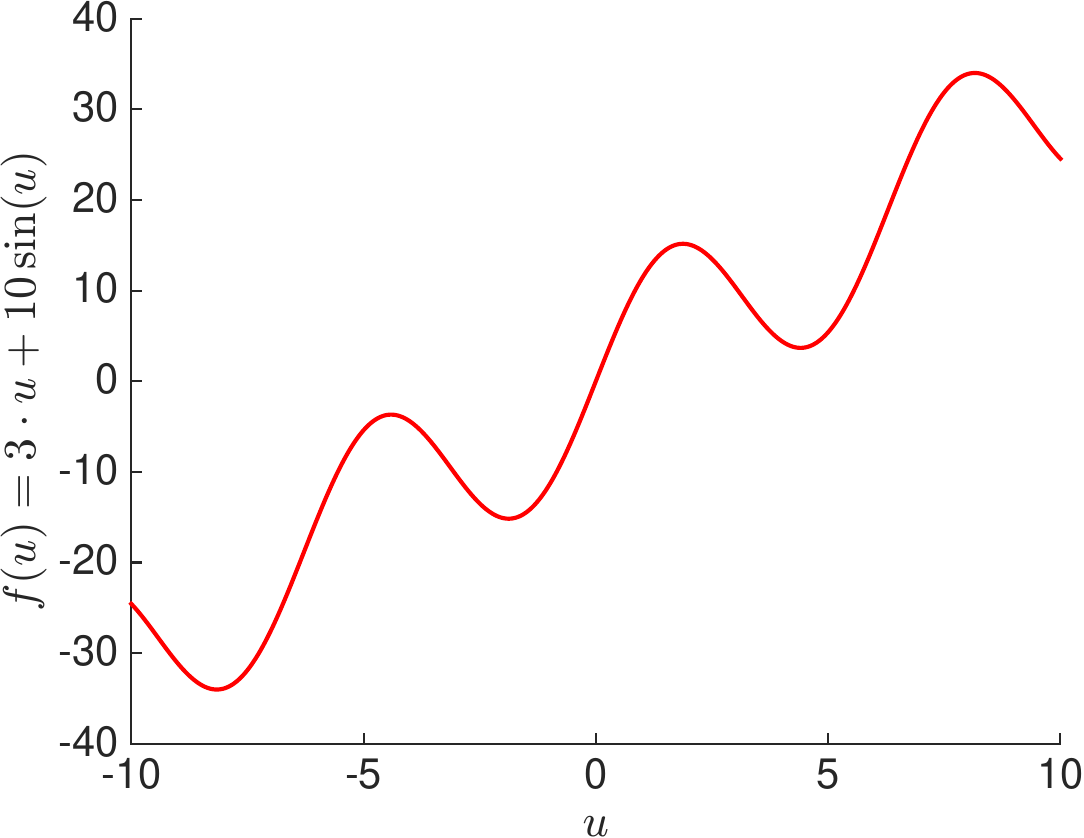}
		&
		\hskip-5pt\includegraphics[width=0.22\textwidth]{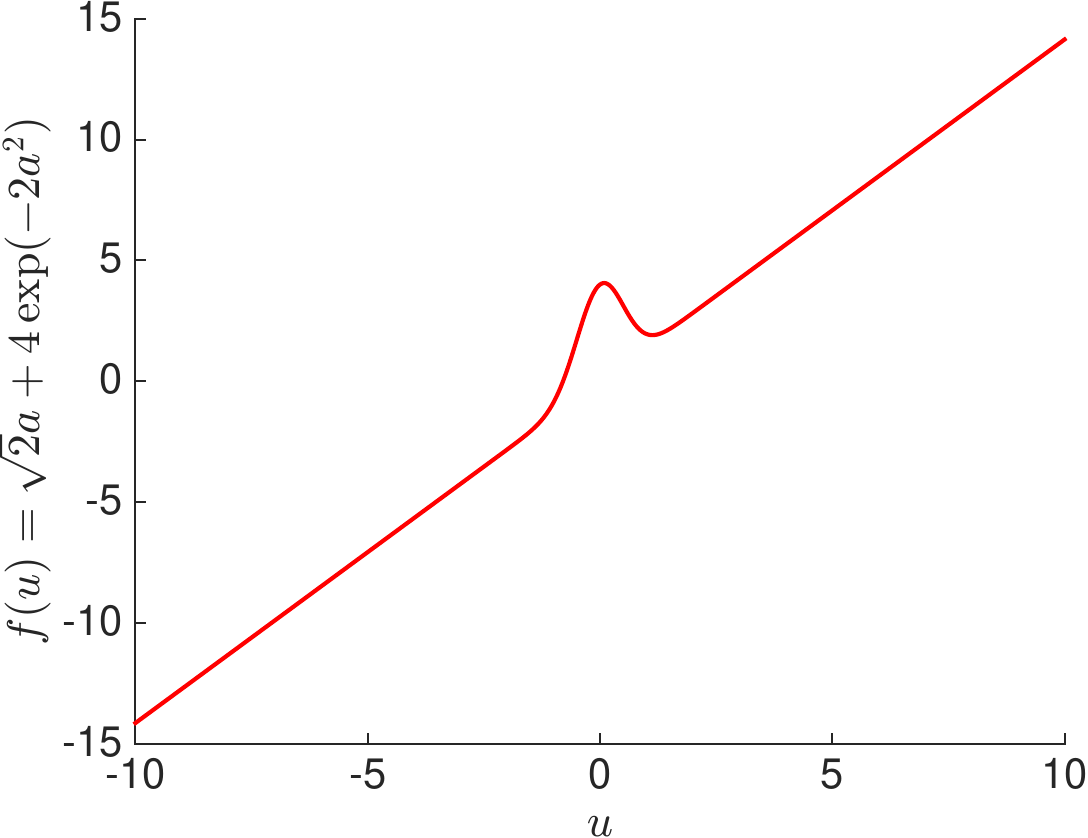}\\
	\end{tabular}\\
\caption{Plot of the link functions $f_1(u) =    3 u + 10 \cdot \sin(u) $ and $f_2 (u) = \sqrt{2} u + 4 \exp(-2 u^2)$.}  
\label{fig:link}
\end{figure}
\begin{figure}[htb]
	\centering
	\begin{tabular}{cc}
		\hskip-27pt\includegraphics[width=0.22\textwidth]{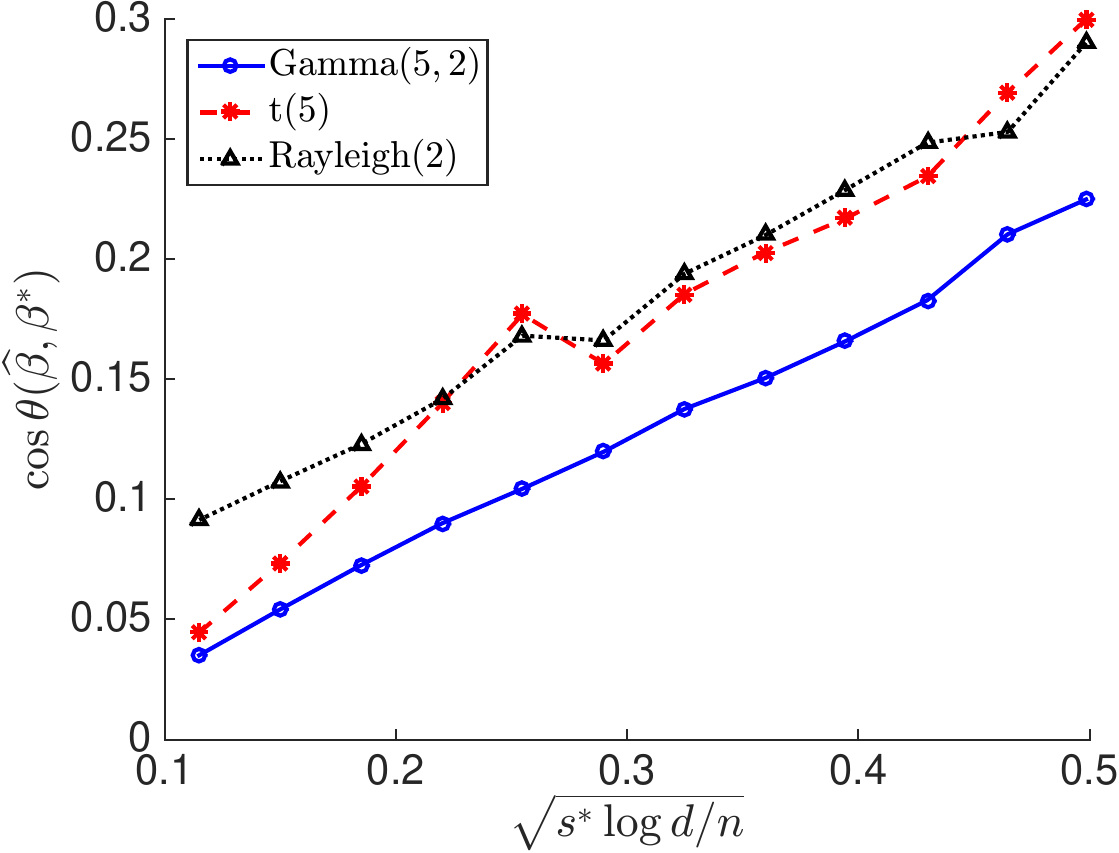}
		&
		\hskip-5pt\includegraphics[width=0.22\textwidth]{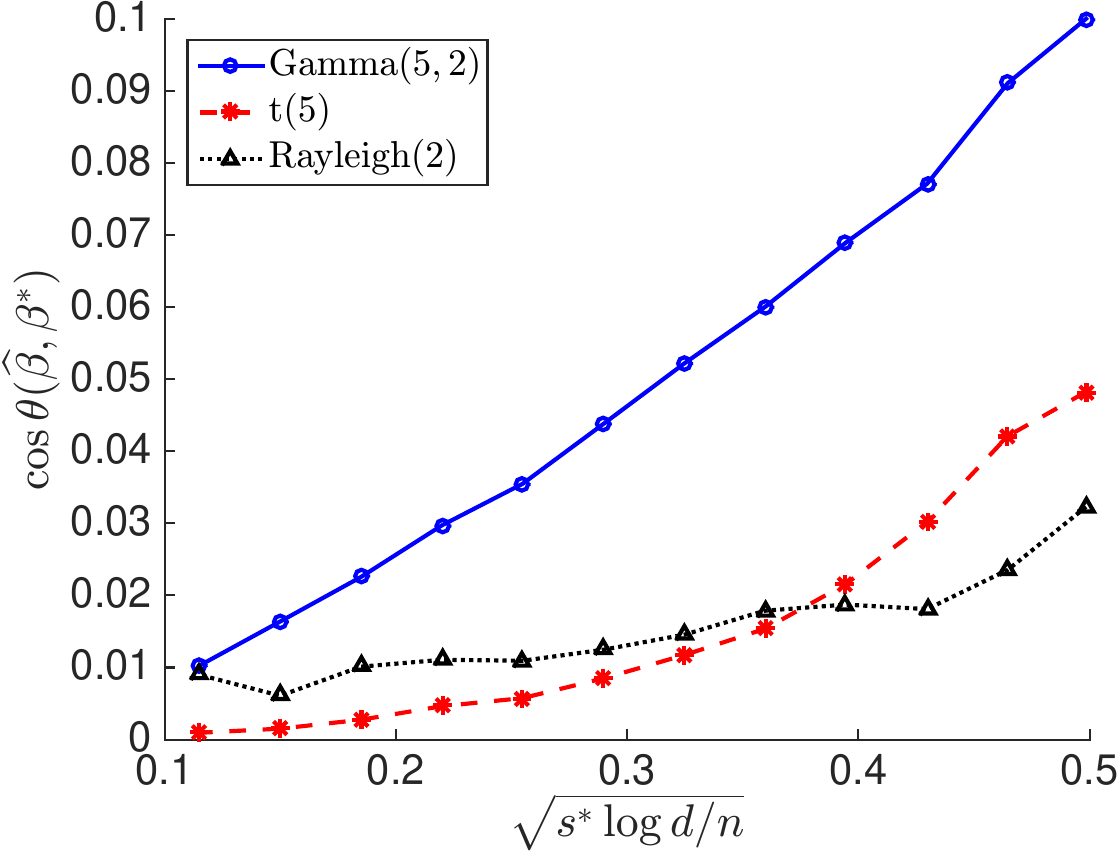}\\
	\end{tabular}\\
	\caption{Cosine distances between the true parameter and the estimated parameter in the sparse SIM with the link function in \eqref{eq:simmodel} set to one of $f_1$ and $f_2$. Here we set  $d = 2000  $. $s^ * = 5$ and vary $n$. }
	\label{fig:simulation_sparse}
\end{figure}
  \begin{figure}[h]
	\centering
	\begin{tabular}{cc}
		\hskip-27pt\includegraphics[width=0.22\textwidth]{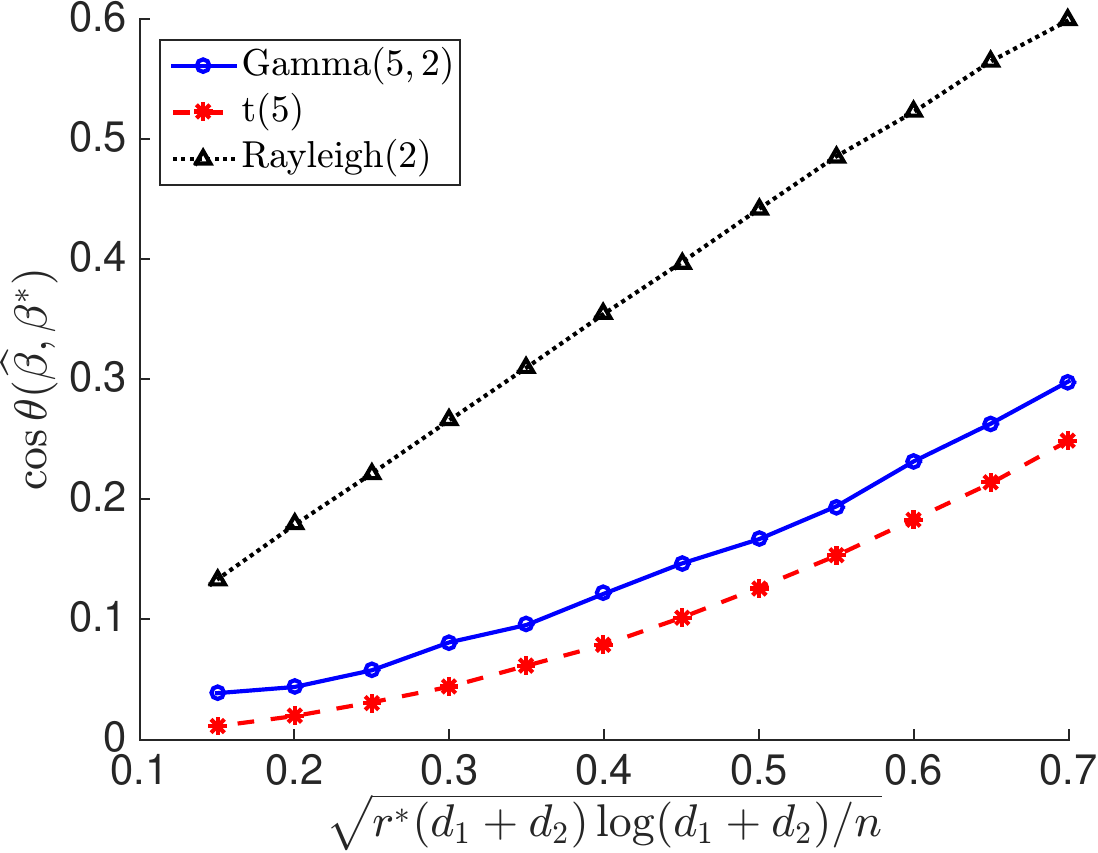}
		&
		\hskip-5pt\includegraphics[width=0.22\textwidth]{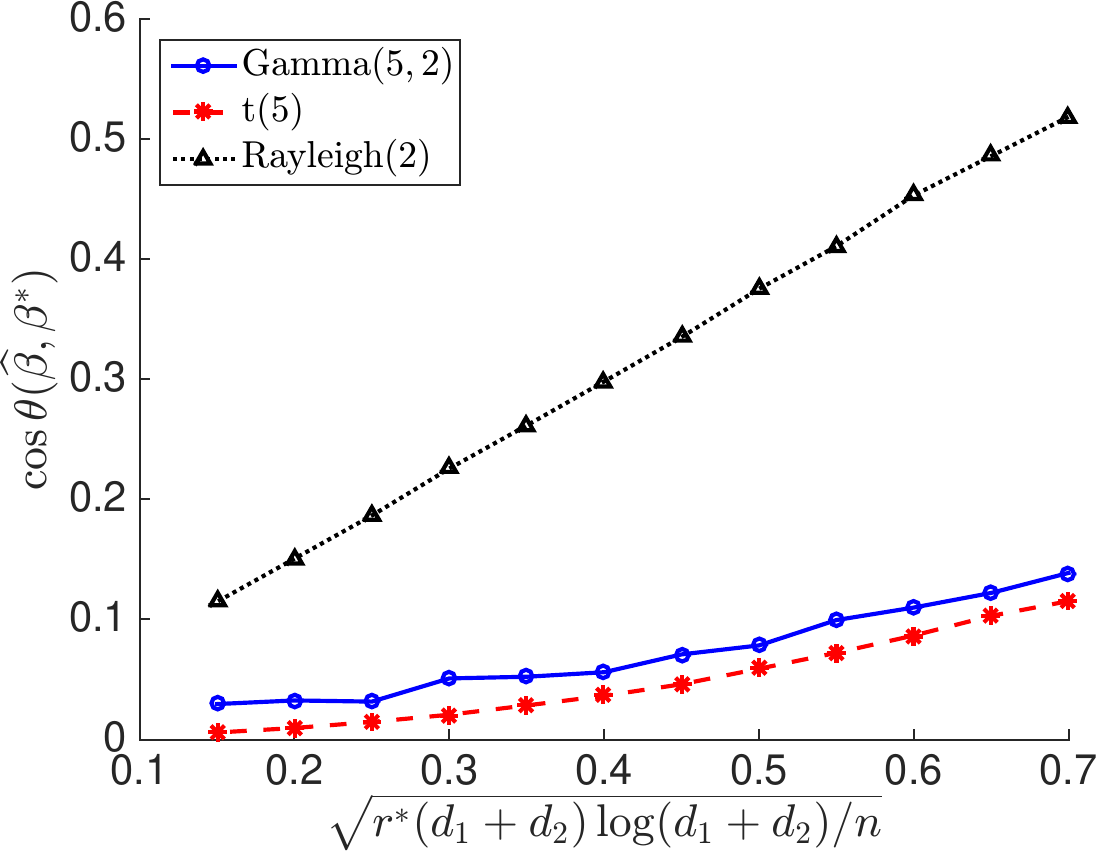}\\
	\end{tabular}\\
	\caption{Cosine distances between the true parameter and the estimated parameter in the low-rank SIM for  with link function in \eqref{eq:simmodel} set to one of $f_1$ and $f_2$. Here we set  $d_1 = d_2 = 20$. $r^* = 3$ and vary $n$.}
	\label{fig:simulation_lowr}
\end{figure}

\textbf{Second-order SIM}: We now concentrate on the problem of sparse phase retrieval using the SDP based estimators proposed based on second-order Stein's identity. Recall that in this case, the link function is known and existing convex and non-convex based estimators are applicable  predominantly for the case of Gaussian or light-tailed data. The question of de-randomization or what are the necessary assumptions on the measurement vectors for (sparse) phase retrieval to work is an intriguing one~\citep{gross2015partial}. Here we demonstrate that using the proposed score-based estimators, one could deal with heavy-tailed and skewed measurement as well, which significantly extend the class of measurement vectors applicable for sparse phase retrieval.
 
Recall that the covariate $X$ has i.i.d. entries with distribution $p_0$. We set $p_0$ to be one of Gamma distribution with shape parameter $5$ and scale parameter  $1$ and  Rayleigh distribution with scale parameter $2$. The random noise $\epsilon$ is set to be standard Gaussian.  Moreover, we solve the optimization problems in \eqref{eq:sparsePCA1} and \eqref{eq:sparsePCA2} via the alternating direction method of multipliers (ADMM)  algorithm proposed in \cite{vu2013fantope}, which introduces a dual variable to encode the constrains and updates the primal and dual variables iteratively.

We set the link function  to be one of $f_3(u) = u^2 $, $f_4 = |u|$, and $f_5 ( u) = 4 u^2 + 3 \cos(u)$. Here $f_3$ corresponds to  the phase retrieval model and $f_4$ and $f_5$ can be viewed as its robust extension. Throughout the experiment we  fix $d = 500$, $s^* = 5$ and vary $n$.  The support of $\beta^*$  is chosen uniformly random among all subsets of $\{1, \ldots, d\}$with cardinality $s^*$. For each $j \in \supp(\beta^*)$, we set $\beta^*_j = 1/\sqrt{s^*} \cdot \gamma_j$, where   $\gamma_j$'s  are  i.i.d. Rademacher random variables.  Furthermore, we fix the regularization parameter $\lambda =4 \sqrt{ \log d/n}$ and  threshold parameter $\tau= 20$.  In addition, we adopt  the cosine distance 
 $ \cos \theta(\hat \beta , \beta^* )  = 1-    | \la \hat \beta, \beta^* \ra |  
$, to measure the estimation error. We plot the cosine distance against the statistical rate of  convergence  $  s ^* \sqrt{ \log d/n}$ in Figure \ref{fig:simulation_quad}-(a)-(c) for each link function, respectively. The plot is based on $100$ independent trials for each $n$, which shows that  the estimation error is bounded by a linear function of $ s ^* \sqrt{ \log d/n}$, which corroborate the theory.
 
\begin{figure}[t]
	\centering
	\begin{tabular}{ccc}
		\hskip-30pt\includegraphics[width=0.15\textwidth]{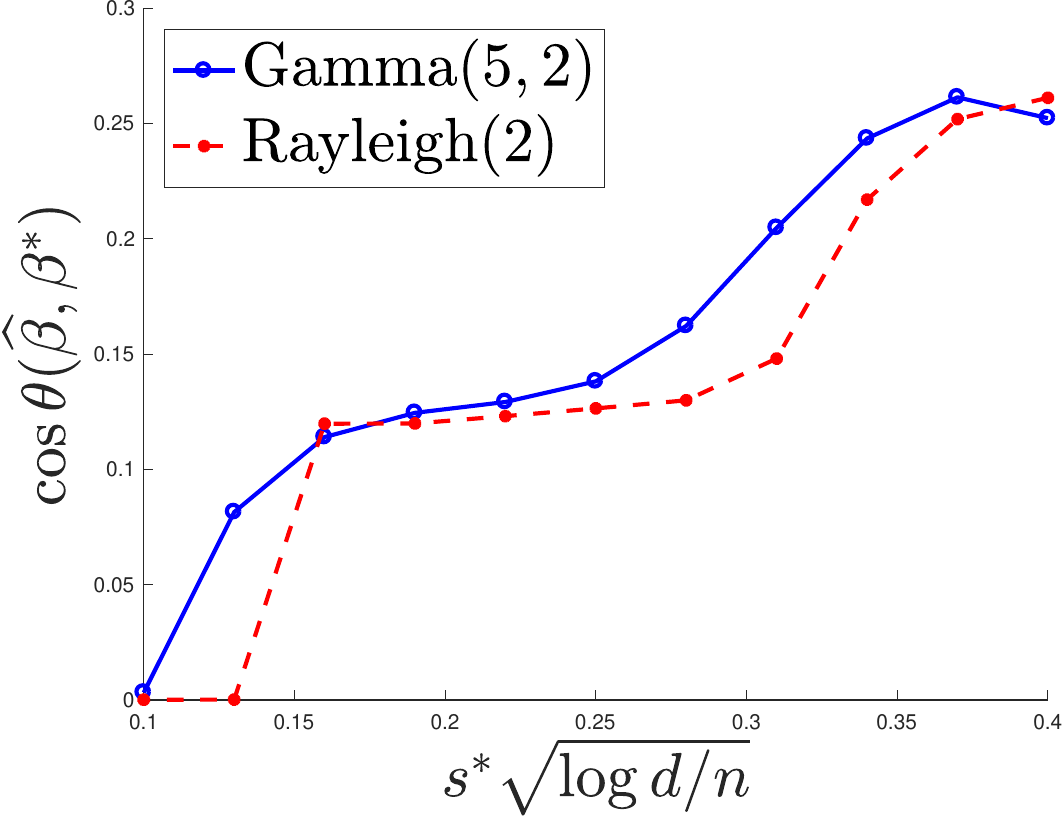}
		&
		\hskip-10pt\includegraphics[width=0.15\textwidth]{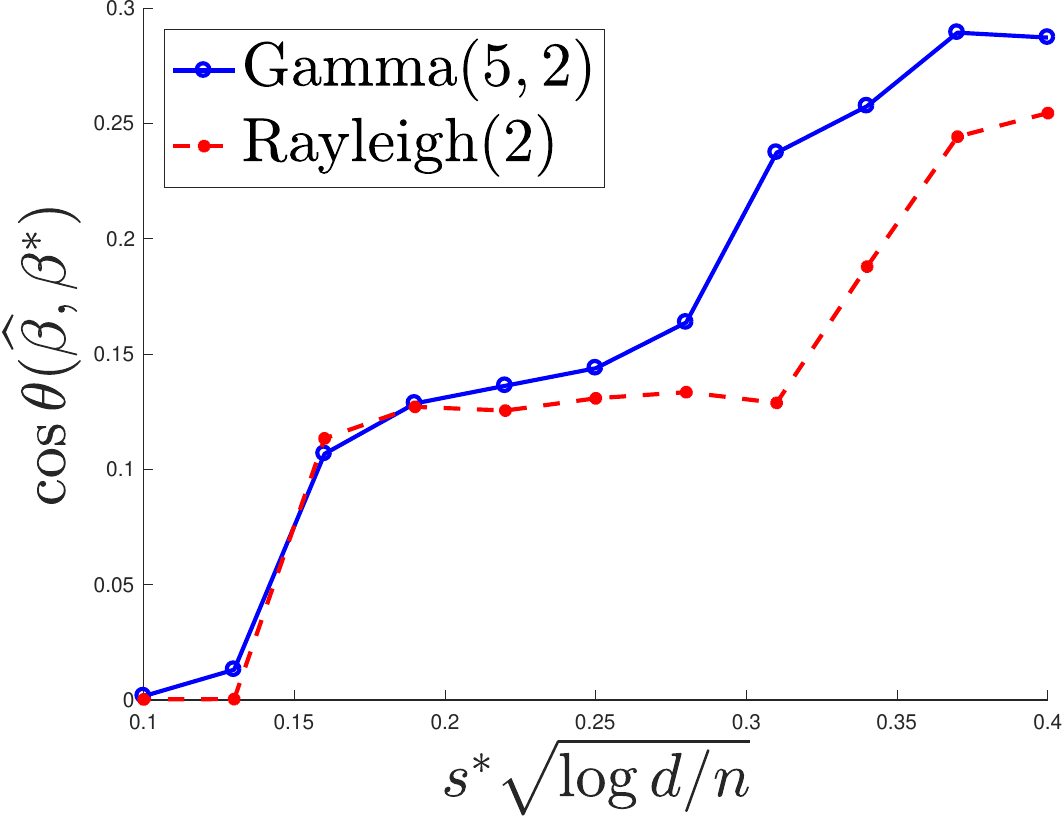} &
		\hskip-5pt\includegraphics[width=0.15\textwidth]{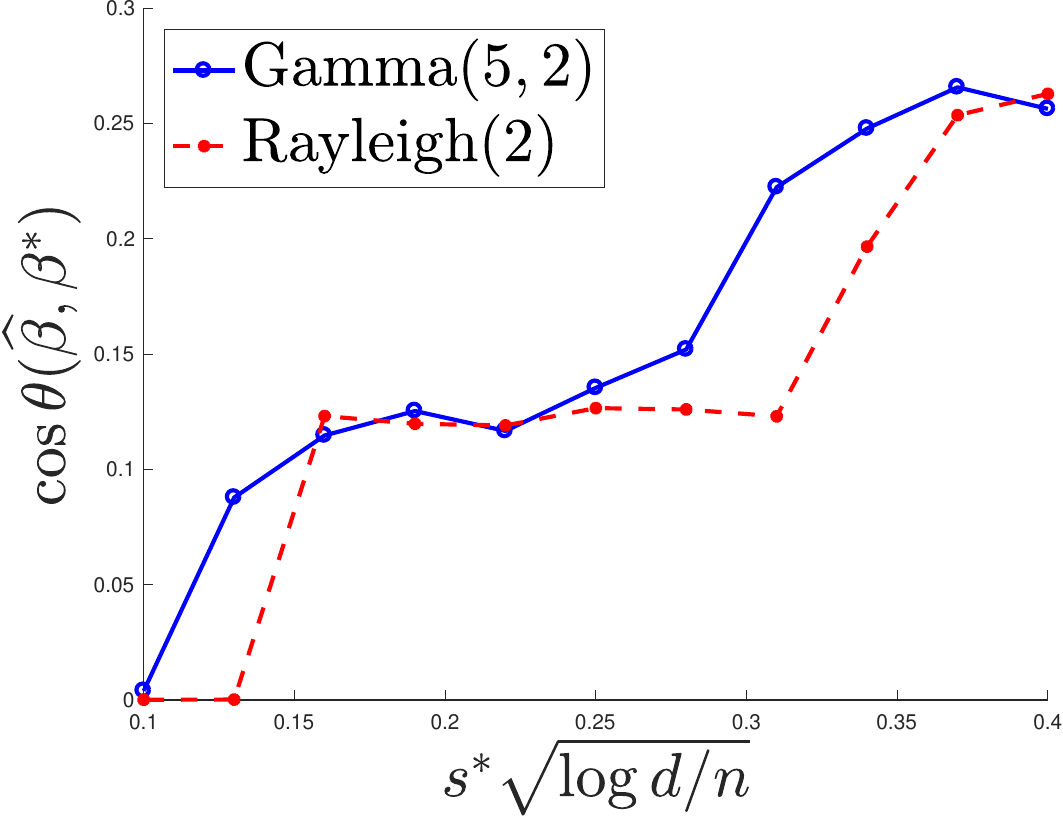} \\
		$f_3(u) = u^2$ & $f_4 (u) = | u|$, &$f_5 ( u) = 4 u^2 + 3 \cos(u)$
	\end{tabular}\\
	\caption{Cosine distances between the true parameter $\beta^*$ and the estimated parameter $\hat \beta$  in the sparse SIM with the link function being    one of $f_3$, $f_4$, and $f_5$. Here we set  $d  = 500  $. $s^ * = 5$ and vary $n$. }
	\label{fig:simulation_quad}
\end{figure}


\section{Conclusion}\label{sec:conclusion}

In this work, we consider estimating the parametric components of single and multiple index models in the high-dimensional setting, under fairly general assumptions on the link function $f$ and response $Y$. Furthermore, our estimators are applicable in the non-Gaussian setting where $X$ is not required to satisfy restrictive Gaussian or elliptical symmetry assumptions. Our estimators are based on a data-driven truncation argument in combination with first and second-order Stein's identity. Furthermore, we show that proposed estimators are near-optimal for several different settings. 
 
 

Recently in the low-dimensional setting, for 2-layer neural networks~\cite{janzamin2015beating} proposed a tensor-based method for estimating the parametric components. Their estimators are sub-optimal even when we consider the low-dimensional Gaussian setting. An immediate application of our truncation based estimators enables us to obtain optimal results for a fairly general class of covariates in the low-dimensional setting. Obtaining similar optimal or near-optimal results in the high-dimensional setting is of great interest for 2-layer neural networks, albeit challenging. We plan to extend the result of this paper for 2-layer neural networks in the high-dimensional setting and report our results in the near future.
\bibliographystyle{plainnat}
\bibliography{main}
\clearpage
\appendix{}

\section{Proofs of the Main Results}
In this section, we lay out the proofs of the theorems in \S\ref{sec:firstordertheory} and \S\ref{sec:secondordertheory}, which establish the statistical rates of convergence of our estimators.
\subsection{Proof of Theorem \ref{thm:sparse}}

\begin{proof}    
Since $\hat \beta $ is the solution of the optimization problem in  \eqref{eq:estimator}, the first-order optimality condition states that 
\#\label{eq:optim_cond}
\nabla L(\hat \beta ) + \lambda \xi  =  0,~~\text{where}~~\xi \in \partial \| \hat \beta \|_1.
\#
Then the entries of $\xi \in \RR^d$ are given by
\$
\xi_j = \sign (\hat \beta_j), ~~\forall j \in \supp(\hat \beta);\\ \xi_j \in [-1, 1], ~~\forall j \notin \supp(\hat \beta).
\$
For any index set $\cA \subseteq [d]$ and $z \in \RR^d$, we define the restriction of $z$ to $\cA$, $z_{\cA} \in \RR^{d}$,  by  letting
\$
[ z_{\cA}] _j  = z_j ~~\text{if}~~j\in \cA, ~~[ z_{\cA}] _j = 0~~\text{otherwise}.
\$
Here $[ z_{\cA}] _j  $ is the $j$-th entry of $z_{\cA}$.
Let $\cS = \supp(\beta^*)$, then
we can write $\xi = \xi_{\cS} + \xi _{\cS^c}$.     For notational simplicity, in the sequel, we define $\theta = \hat \beta - \mu \cdot \beta ^*$.   Thus by \eqref{eq:optim_cond} it holds that 
\#\label{eq:upper_quad1} 
& \la \nabla L(\hat \beta ) - \nabla L(\mu\beta^* ),\theta  \ra  = \la - \lambda \cdot \xi - \nabla L(\mu\beta^* ), \theta \ra  \notag \\
& \qquad \leq \la - \lambda \cdot \xi_{\cS} - \lambda \cdot \xi_{\cS^c} ,\theta \ra + \| \nabla L (\mu \beta ^*) \|_{\infty} \cdot \| \theta   \|_1.
\#
By the definition of $\xi$, we have \#\label{eq:upper2}
\la - \lambda \cdot \xi_{\cS^c} ,\hat \beta - \mu \beta ^* \ra = - \lambda \cdot \| \hat\beta_{\cS^c} \|_1.
\#  Moreover, since $\| \xi \|_{\infty} \leq 1$, H\"older's inequality implies that 
\#\label{eq:upper3}
\la - \lambda \cdot \xi_{\cS} , \theta \ra \leq \lambda \cdot  \| \theta _{\cS } \|_1.
\#
Note that $\nabla ^2 L(\beta) = 2 I_d$. Combining \eqref{eq:upper_quad1}, \eqref{eq:upper2}, and \eqref{eq:upper3}, we obtain
\#\label{eq:upper4}
2 \| \theta \|_2 ^2 &= \la \nabla L(\hat \beta ) - \nabla L(\mu\beta^* ),\theta  \ra  \\ 
&\leq - \lambda \| \theta_{\cS^c} \|_1 + \lambda \| \theta _{\cS} \|_1  + \| \nabla L (\mu \beta ^*) \|_{\infty} \cdot \| \theta   \|_1.
\#
For an upper bound of the right-hand side of \eqref{eq:upper4}, we apply the following lemma to obtain an upper bound on $\| \nabla L (\mu \beta ^*) \|_{\infty} $. 
\begin{lemma} [Bound on $\| \nabla L (\mu \beta ^*) \|_{\infty} $] \label{lemma:grad_loss1}
We set the truncation level in \eqref{eq:thresholding}  as   $ \tau = 2  (   M \cdot n / \log d )^{1/4} $. Then we have
\$
\PP \Bigl [   \| \nabla    L (\mu \beta ^*)  \|_{\infty}   >  7 \sqrt{ M \cdot \log d/ n} \Bigr ] \leq d^{-2}.
\$
\end{lemma}
\begin{proof}
See \S \ref{sec:proof:lemmaLgrad_loss1} for a detailed proof.
\end{proof}
Thus by Lemma \ref{lemma:grad_loss1} and the choice of $\lambda$, we have 
  $\lambda > 2 \| \nabla L (\mu \beta ^*) \|_{\infty} $ with probability at least $1- d^{-2}$.  This implies that  
\#\label{eq:upper5}
2 \| \theta \|_2 ^2 \leq - \lambda /2 \cdot \| \theta_{\cS^c} \|_1  + 3 \lambda/2 \cdot \| \theta _{\cS} \|_1 \leq 2 \lambda \cdot  \| \theta_{\cS}\|_1.
\#
Since the leftmost term in \eqref{eq:upper5} is nonnegative, we obtain
  $ \| \theta _{\cS^c} \|_1 \leq 3\cdot  \| \theta _{\cS} \|_1 $.  In addition, since  $|\cS| = s^*$, $\| \theta_{\cS}\|_1 \leq \sqrt{s^* }  \cdot \| \theta_{\cS}\|_2$. Thus by  \eqref{eq:upper5} we have $ \| \theta  \| _2\leq \sqrt{s^*} \cdot \lambda $.   Moreover, we also have $\| \theta_{\cS}\|_1 \leq s^* \lambda$, which further implies  that 
\$
\| \theta \|_1 = \| \theta _{\cS} \|_1 + \| \theta _{\cS^c} \|_1 \leq 4 \cdot \| \theta _{\cS} \|_1 \leq 4 s ^* \lambda.
\$
Therefore, we conclude the proof.
\end{proof}

\subsection{Proof of Theorem \ref{thm:lowrank}}\label{app:lr}
\begin{proof}
The proof of Theorem \ref{thm:lowrank} is parallel to that of Theorem \ref{thm:sparse}. Here the difference is to handle the nuclear norm regularization, instead of the $\ell_1$-penalty.
Since $\hat \beta  $ is the solution of the optimization problem in  \eqref{eq:estimator}, the first order optimality condition states that 
\#\label{eq:optim}
  L(\hat \beta ) + \lambda \| \hat \beta \|_{\nuc} \leq   L(\mu \beta^* )  + \lambda \| \mu \beta^* \| _{\nuc}.
\#
To simplify the  notation, we define $\Theta = \hat \beta - \mu \cdot \beta ^*$. Since $ L $ is quadratic, 
\#\label{eq:quad_loss}
L( \hat \beta ) - L(\mu \beta^* ) = \la \nabla L(\mu \beta^*), \Theta \ra  + 2 \| \Theta \|_{\fro}^2, 
\#
where $\nabla L$ takes values in $\RR^{d_1 \times d_2}$.
Then combining \eqref{eq:optim}, \eqref{eq:quad_loss}, and H\"older's inequality, we have   
\#\label{eq:upper_quad1}
2 \| \Theta   \|_{\fro}^2  & \leq - \la \nabla L(\mu \beta^*), \Theta \ra    + \lambda \|  \mu \beta^* \|_{\nuc} - \lambda \| \hat \beta \|_{\nuc} \notag \\ & \leq \bigl \| \nabla L(\mu \beta^*) \bigr \|_{\oper} \cdot \| \Theta \|_{\nuc} +     \lambda \|  \mu \beta^* \|_{\nuc} - \lambda \| \hat \beta \|_{\nuc}.
\# 
In the following, we focus on the term $\| \mu \beta^* \|_{\nuc} - \| \hat \beta \|_{\nuc}$ in \eqref{eq:upper_quad1}. 
Let $  U  \Lambda^* V^\top $ be the singular value decomposition of $\mu \beta^*$, where 
$U\in \RR^{d_1\times d_1} $ and $V \in \RR^{d_2 \times d_2} $ are orthogonal matrices, and $\Lambda^* \in \RR^{d_1 \times d_2}$ be  formed by the singular values of $\mu\beta^*$.  Moreover, since $ \rank(\beta^* )= r^*$, $\Lambda ^*$ can be written~in block form~as
\#\label{eq:svd_block}
\Lambda ^* = \begin{bmatrix} 
\Lambda^*_{11} & 0 \\
0 & 0 
\end{bmatrix},
\#
where $\Lambda^*_{11} \in \RR^{r^*\times r^*}$ is a diagonal matrix whose diagonal elements are the nonzero singular values of $\mu \beta^*$.  We define $\Gamma = U^{\top} \Theta V $, which can be written in block form as 
\$
\Gamma = \begin{bmatrix}
\Gamma_{11} & \Gamma _{12} \\
\Gamma_{21} & \Gamma_{22} 
\end{bmatrix},
\$ 
where $\Gamma _{11} \in \RR^{r^* \times r^*}$. In addition, we define matrices 
\$
\Gamma ^{(1)} = \begin{bmatrix}
0& 0\\
0 & \Gamma_{22} 
\end{bmatrix} ~~\text{and}~~\Gamma^{(2)} = \begin{bmatrix}
\Gamma_{11} & \Gamma _{12} \\
\Gamma_{21} & 0
\end{bmatrix}.
\$
Then by \eqref{eq:svd_block} and   triangle inequality of the nuclear norm, we have 
\#\label{eq:upper_quad2}
& \| \hat \beta \|_{\nuc} =  \|  \mu \beta^* + \Theta \|_{\nuc} = \| U (\Lambda ^* + \Gamma) V^{\top} \|_{\nuc}     \notag \\
& \qquad  = \| \Lambda ^* + \Gamma \|_{\nuc}
   \geq  \| \Lambda ^* + \Gamma^{(1)} \|_{\nuc} - \| \Gamma ^{(2)} \|_{\nuc}\notag \\
   &\qquad  = \| \Lambda^* \|_{\nuc} + \| \Gamma ^{(1)} \|_{\nuc} - \| \Gamma ^{(2)} \|_{\nuc},
\#
where the last equality follows from the fact that $\Lambda ^* + \Gamma^{(1)} $ is block diagonal.
Since $\| \mu \beta ^* \|_{\nuc} = \| \Lambda ^* \|_{\nuc}$, by \eqref{eq:upper_quad2} we obtain
\#\label{eq:upper_quad3}
\| \mu \beta^* \|_{\nuc} - \| \hat \beta \|_{\nuc} \leq  \| \Gamma ^{(2)} \|_{\nuc} - \| \Gamma ^{(1)} \|_{\nuc}.
\#
In addition, triangle inequality implies that 
\#\label{eq:upper_quad4}
\| \Theta \|_{\nuc} = \| U \Gamma  V ^\top \|_{\nuc} \leq \| \Gamma^{(1)} \|_{\nuc} + \| \Gamma^{(2)} \|_{\nuc}.
\#
Thus combining \eqref{eq:upper_quad2}, \eqref{eq:upper_quad3}, \eqref{eq:upper_quad4},  we have 
\#\label{eq:upper_quad5}
   2 \| \Theta   \|_{\fro}^2  &  \leq     \bigl ( \bigl \| \nabla L(\mu \beta^*) \bigr \|_{\oper} +  \lambda  \bigr )  \cdot  \| \Gamma ^{(2)} \|_{\nuc}  \nonumber \\ 
 & \qquad +  \bigl ( \bigl \| \nabla L(\mu \beta^*) \bigr \|_{\oper} - \lambda \bigr )\cdot  \| \Gamma^{(1)}\|_{\nuc}  .
\#
We utilize the following lemma to obtain an upper bound of  $\| \nabla L (\mu \beta ^*) \|_{\oper} $.

\begin{lemma} [Upper bound of $\| \nabla L (\mu \beta ^*) \|_{\oper} $] 
\label{lem:grad_loss2}
Let  $L\colon \RR^{d_1\times d_2} \rightarrow \RR$ be the loss function defined in \eqref{eq:define_a_loss} for the matrix setting. Setting   $$ \kappa = 2  \sqrt{n \cdot \log (d_1 + d_2)}  /\sqrt{  (d_1 + d_2 )M},$$  then it holds that 
\$
\PP \Bigl [ \| \nabla L (\mu \beta ^*) \|_{\oper}  > 6 \sqrt{ (d_1+d_2) /n } \Bigr ] \leq (d_1 + d_2)^{-2}.
\$
\end{lemma}
\begin{proof}
See \S\ref{sec:proof:lemmaLgrad_loss2} for a detailed proof.
\end{proof}
By Lemma \ref{lem:grad_loss2} and the choice of $\lambda$, we conclude that        $\lambda > 2 \cdot \| \nabla L (\mu \beta ^*) \|_{\oper} $ with  probability at least $1-(d_1+ d_2)^{-2}$.  Thus by \eqref{eq:upper_quad5} we have
\#\label{eq:upper_quad6}
2 \| \Theta   \|_{\fro}^2 \leq 3 \lambda  / 2 \cdot \| \Gamma ^{(2)} \|_{\nuc} - \lambda /2 \cdot \| \Gamma^{(1)} \|_{\nuc} 
\#
which implies that $\| \Gamma^{(1)} \|_{\nuc}   \leq 3\cdot \| \Gamma ^{(2)}  \|_{\nuc} $. Moreover, by the subadditivity of rank, we obtain 
\$
\rank ( \Gamma ^{(2)}   )&  \leq \rank  \biggl ( \begin{bmatrix}
\Gamma_{11} /2 & \Gamma _{12} \\
0  & 0
\end{bmatrix} \biggr )  \\
& \qquad  + \rank  \biggl ( \begin{bmatrix}
\Gamma_{11} /2 & 0 \\
\Gamma_{21} & 0
\end{bmatrix} \biggr ) = 2r^*,
\$
which implies that  $ \| \Gamma ^{(2)}  \|_{\nuc} \leq \sqrt{ 2r^* } \cdot \| \Gamma ^{(2)} \|_{\fro}$
Then by \eqref{eq:upper_quad6} we obtain that $ \| \Theta  \| _{\fro} \leq 3 /\sqrt{2} \cdot \sqrt{r^*} \cdot \lambda $.  Finally, by triangle inequality for the nuclear norm, 
\$
\| \Theta  \|_{\nuc}  &= \| \Gamma  \|_{\nuc} \leq \| \Gamma^{(1)} \|_{\nuc} + \| \Gamma^{(2)} \|_{\nuc}  \\ &\leq 4\cdot  \| \Gamma^{(2)} \|_{\nuc} \leq 4  \sqrt{2 r^*}  \| \Gamma ^{(2)} \|_{\fro}  = 12 r^* \lambda.
\$
Thus we conclude the proof of Theorem \ref{thm:lowrank}.
\end{proof}

\subsection{Proof of Theorem \ref{thm:pr1}}
\begin{proof}
	We denote by $\hat W$ the solution of the optimization problem in \eqref{eq:sparsePCA1}. In addition, we let  $W^* = \betas {\betas }^\top $. In the following, we  establish an upper bound for $\| \hat W - W ^* \|_{\oper}$.

	Since $W^*$ is feasible for the optimization problem in \eqref{eq:sparsePCA1}, we have
	\#\label{eq:feasible1}	  \la \hat W , \tilde \Sigma \ra -  \lambda \| \hat W \|_{1} \geq
	 \la W^* , \tilde \Sigma \ra - \lambda \| W^* \|_{1}.
	 \#
	 We denote $\Sigma ^*  = \EE [ Y \cdot T (X) ] $. Note that $\betas$ is the leading eigenvector of $\Sigma^*$. Then \eqref{eq:feasible1} is equivalent to
	\# \label{eq:feasible2}
& 	\la \hat W - W^* , \tilde \Sigma - \Sigma^*  \ra - \lambda   \| \hat W \|_{1} +  \lambda \| W^* \|_1 \notag \\
	&\qquad    \geq \la  \Sigma^* , W^* - \hat W   \ra  .
\#
The following Lemma in \cite{vu2013fantope} (Lemma 3.1) establishes an upper bound for the first term on the left-hand side of \eqref{eq:feasible2}.

\begin{lemma} \label{lemma:curvature}
Let $\Omega \in \RR^{d  \times d} $ be a symmetric matrix and let $\lambda_1 \geq \lambda_2 \geq \ldots \lambda _d $ be the  eigenvalues of $\Omega$ in the  descending order. For any $\ell \in [d-1]$ such that $\lambda_\ell - \lambda_{\ell +1} >0$, let $\Pi_\ell \in \RR^{d\times d} $ be the projection matrix for the subspace spanned by the eigenvectors of $\Omega$ corresponding to $\lambda_1, \ldots, \lambda_{\ell}$. Then for any $\Lambda \in\RR^{d \times d}  $ satisfying $0 \preceq \Lambda\preceq I_d$ and $\trace(\Lambda) = \ell$, we have
\$
(\lambda_\ell - \lambda_{\ell +1} ) \cdot \| \Pi_k - \Lambda \|_{\fro}^2 \leq 2 \la \Omega, \Pi_\ell - \Lambda \ra.
\$
\end{lemma}
 Note that $W^*$ is the projection matrix for the subspace spanned by $\betas$.
Applying  Lemma \ref{lemma:curvature} to $\Sigma^*$ with $\ell = 1$, we have
\#\label{eq:apply_lemma_curvature}
\la \Sigma^* , W^* - \hat W \ra \geq C_0 /2 \cdot \| \hat W - W ^* \|_{\fro }^2,
\#
where $C_0 > 0$ is defined in \eqref{eq:apply2stein}. In addition, by H\"older's inequality, we have
\#\label{eq:apply_holder}
& \la \hat W - W^* , \tilde \Sigma - \Sigma^*  \ra \notag \\
&\qquad  \leq \| \tilde \Sigma  - \Sigma ^* \|_{\infty} \cdot \| \hat W - W^* \|_{1} .
\#
In what follows, we bound $\| \tilde \Sigma - \Sigma ^* \|_{\infty}$.
  \begin{lemma}\label{lemma:truncation2}
  Let $\tilde \Sigma$ be defined in \eqref{eq:estimate_cov} and we define $  \Sigma^*   = \EE [ Y \cdot T(X)]$. Under Assumption \ref{assume:moments}, for any  truncation level $\tau >0 $ in \eqref{eq:truncation}, with probablity at least $1- d^{-2}$, we have
  \#\label{eq:truncation_result}
 & \| \tilde \Sigma  - \Sigma ^* \|_{\infty}  \leq 9 M \cdot \tau^{-3}  \notag  \\
 &  \qquad \qquad +2  \tau^{3} \cdot \log d /n + 2\sqrt{   5M \cdot  \log d /n}.
  \#
    \end{lemma}
  \begin{proof}
  See \S \ref{proof:lemma:truncation2}  for a detailed proof.
  \end{proof}
By this lemma, if we set $\tau = ( 1.5 M n / \log d)^{1/6}$, then with probability at least $1- d^{-2}$,
\#\label{eq:seperation}
\| \tilde \Sigma - \Sigma^* \|_{\infty}  & \leq ( 2 \sqrt{5} + 2\sqrt{6})\cdot \sqrt{M \log d / n} \\ &\leq 10 \sqrt{M \log d /n}.
\#
Thus by setting $\lambda = 10 \sqrt{ M \log d /n }  $ we have $\| \tilde \Sigma - \Sigma ^* \|_{\infty} \leq \lambda $ with probability at least $1- d^{-2} $.

Then combining \eqref{eq:feasible2}, \eqref{eq:apply_lemma_curvature}, and \eqref{eq:apply_holder} we have
\#\label{eq:bound_l1}
\lambda  \left ( \| \hat W - W^* \|_{1} -   \| \hat W \|_{1}  +  \| W^* \|_1  \right ) \nonumber \\ \geq  C_0 /2 \cdot \| \hat W - W ^* \|_{\fro }^2.
\#
Note that $W^* = \betas {\betas}^\top$ and that $\betas$ is $s^*$-sparse. We denote the support of $W^*$ by $\cJ$, which is given by
\$
\cJ = \left \{ (j,k ) \in [d] \times [d] \colon \beta_j^* \cdot \beta_k ^* \neq 0 \right \}.
\$
Then by separation of the $\ell_1$-norm, we have
\$
\| \hat W \|_{1} &= \| \hat W_{\cJ} \|_1 + \| \hat W _{\cJ^c} \|_1 \\ \nonumber
\| \hat W - W^* \|_1 &= \| \hat W _{\cJ} - W^* _{\cJ} \|_1 + \| \hat W _{\cJ^c} \|_1 ,
\$
which implies that
\#\label{eq:l1_norm}
&\| \hat W - W^* \|_{1} -   \| \hat W \|_{1} +  \| W^* \|_1 \notag \\ 
& \qquad =   \| \hat W_{\cJ} - W^*_\cJ \|_{1} -  \| \hat W_{\cJ} \|_1  + \| W^* _{\cJ}\|_1 \notag \\
& \qquad  \leq   2 \| \hat W_{\cJ} - W^*_\cJ\|_{1} \leq 2 {s^*}^2  \| \hat W - W ^* \|_{\fro}.
\#
Here the last inequality in \eqref{eq:l1_norm} follows from the fact that $|\cJ | = {s^*}^2. $
Combining \eqref{eq:bound_l1} and \eqref{eq:l1_norm}, we obtain
\#\label{eq:initialization}
 \| \hat W - W^* \|_{\fro} \leq 4/ C_0 \cdot s^* \lambda.
\#
Since $\hat \beta  $ is the leading eigenvector of $\hat W$, we have $\| \hat \beta  - \betas \|_2 \leq \sqrt{2} \| \hat W - W^* \|_{\fro} \leq 4 \sqrt{2} / C_0 \cdot s^* \lambda$, which concludes the proof.
\end{proof}

\subsection{Proof of Theorem \ref{thm:pr2}}
\begin{proof}
The proof is similar to that of Theorem \ref{thm:pr1}. In the case of sparse MIM, we denote $W^* = B^* {B^*}^{\top}$. Note that $\hat W$ is the solution to the optimization problem in \eqref{eq:sparsePCA2} and that  $\hat B$ consists of the top-$k$ eigenvectors of $\hat W$. Then by Corollary 3.2 in \cite{vu2013fantope},  we have
\#\label{eq:first_step_bound}
\inf_{O \in \mathbb{O}_k } \| \hat B - B^* O \|_{\fro} \leq \sqrt{2} \| \hat W - W^* \|_{\fro} .
\#
In what follows, we derive an upper bound for $\hat W - W^*$. Note that since $B^*$ is orthonormal, $\trace(W^*) = k$. Thus $W^*$ is feasible for \eqref{eq:sparsePCA2}, which implies
\# \label{eq:feasible22}
\la \hat W - W^* , \tilde \Sigma - \Sigma^*  \ra - \lambda   \| \hat W \|_{1}  + \lambda \| W^* \|_1   \geq  \notag \\\la  \Sigma^* , W^* - \hat W   \ra.
\#
Here we define  $\Sigma^* = \EE [ Y \cdot T(X)]$. Note that $W^*$ is the projection matrix for the subspace spanned by the top-$k$ leading eigenvectors of $\Sigma^*$. By Lemma \ref{lemma:curvature} with $\ell = k$, we have
\$
\la  \Sigma^* , W^* - \hat W   \ra  \geq \rho_0 / 2 \cdot \| \hat W - W^*\|_{\fro }^2 ,
\$
where $\rho_0$ is the smallest eigenvalue of $\EE [ \nabla ^2 f ( X B^*)]$. Similar to the proof of Theorem \ref{thm:pr1}, by H\"older's inequality and \eqref{eq:feasible22}, we have
\#\label{eq:some_inter}
\|
&  \tilde \Sigma - \Sigma^* \|_{\infty} \cdot \| \hat W -W^* \|_1 - \lambda \| \hat W \|_1+ \lambda \|W^* \|_1 \notag
 \\ 
 & \qquad  \geq
\rho_0 /2 \cdot \| \hat W - W^* \|_{\fro}^2 .
\#
By Lemma  \ref{lemma:truncation2}, if we set $\lambda = 10 \sqrt{ M \log d /n}$, with probability at least $1 - d^{-2}$, we have
\#\label{eq:set_regparam}
\| \hat \Sigma - \Sigma ^* \|_{\infty} \leq \lambda.
\#
 Note that the support of $W^*$ is
\$
\cJ \subseteq  \left \{ (j,k) \in [d]\times [d] \colon \|B_{j\cdot}^* \|_2 \cdot \|B_{k\cdot}^* \|_2\neq  0 \right \}.
\$
Since $B^*$ is $s^*$-row sparse, $| \cJ| \leq {s^*}^2$. Thus      \eqref{eq:l1_norm} also hold for the MIM. Combining \eqref{eq:some_inter}, \eqref{eq:set_regparam}, and \eqref{eq:l1_norm}, we   obtain
\#\label{eq:second_step_bound}
\| \hat W - W^*\|_{\fro} \leq  4  / \rho_0 \cdot s^* \lambda.
\#
Finally, combining \eqref{eq:first_step_bound} and \eqref{eq:second_step_bound}, we conclude the proof.
\end{proof}

\section{Proof of Auxiliary Results}
\subsection{Proof of Lemma \ref{lemma:grad_loss1}} \label{sec:proof:lemmaLgrad_loss1}
\begin{proof}
By definition of the loss function $L$ in \eqref{eq:estimator}, we have 
\$
\nabla    L (\mu \beta ^*) & =  2  \mu \beta^* - \frac{2}{n} \sum_{i=1}^n  \tilde Y_i \cdot \tilde S(X_i) \notag\\ &  =    \EE \bigl [2  Y_i \cdot S(X_i) \bigr ] - \frac{2}{n} \sum_{i=1}^n  \tilde Y_i \cdot \tilde S(X_i) .
\$
By triangle inequality, 
\#\label{eq:grad1}
 \|  \nabla    L (\mu \beta ^*) \|_{\infty} \notag  &\leq  \Bigl \| \EE \bigl [2  Y \cdot S(X) \bigr ] - \EE \bigl [ 2 \tilde Y \cdot \tilde S(X) \bigr ] \Bigr  \|_{\infty} \notag \\ &\qquad + \biggl \| \EE \bigl [ 2 \tilde Y \cdot \tilde S(X) \bigr ] - \frac{2}{n} \sum_{i=1}^n  \tilde Y_i \cdot \tilde S(X_i) \biggr \|_{\infty}.
\#
For any $j \in [d]$, by the definition of the truncated response $\tilde Y$ and truncated score $\tilde S$, we obtain 
\#\label{eq:grad2}
& \bigl | \EE \bigl [  \tilde Y  \cdot \tilde S_j(X) \bigr ]  - \EE \bigl [ Y \cdot S_j(X) \bigr ] \bigr |  \notag  \\ 
& \leq 
\Bigl |  \EE \Bigl \{ \tilde Y  \cdot \bigl [   \tilde S_j(X)  - S_j (X) \bigl ]\Bigr\} \Bigr | +  \bigl |  \EE \bigl [   ( \tilde Y - Y)  \cdot   S_j (X) \bigl ]  \bigr | \notag \\
&= \underbrace{\bigl |  \EE \bigl [  \tilde Y  \cdot      S_j (X) \cdot \ind \{ |S_j (X) | >  \tau \} \bigr] \bigr |}_{a_1} \notag \\
& \qquad + \underbrace{\bigl |  \EE \bigl [   Y  \cdot      S_j (X) \cdot \ind \{ |Y | >  \tau  \} \bigr] \bigr |}_{a_2}.
\#
By Cauchy-Schwarz inequality, we have
\#\label{eq:grad3}
a_1^2 & \leq \EE \bigl [ \tilde Y^2 S_j^2 (X)  \bigr ]\cdot \PP \bigl [ |S_j(X) | \geq \tau   \bigr ] \notag \\
& \leq  \sqrt{ \EE  (\tilde Y^4) \cdot \EE \bigl [ S_j^4 (X)  \bigr ]} \cdot \EE \bigl [ S_j^4 (X) \bigr ]   \cdot  \tau ^{- 4} \notag \\
& = M^2\cdot \tau ^{-4},
\#
where the second inequality follows from Chebyshev's inequality. Similarly, for $a_2$ we have
\#\label{eq:grad4}
a_2 ^2 &  \leq \EE \bigl [   Y^2 S_j^2 (X)  \bigr ]\cdot \PP \bigl ( |Y  | \geq \tau \bigr ) \notag \\
&\leq  \sqrt{ \EE  (\tilde Y^4) \cdot \EE \bigl [ S_j^4 (X)  \bigr ]} \cdot \EE (Y^4 ) \cdot \tau^{-4}  \notag \\
& \leq M^2  \cdot \tau^{-4}.
\#
Thus combining \eqref{eq:grad2}, \eqref{eq:grad3}, and \eqref{eq:grad4}, we conclude that  
\$
\Bigl | \EE \bigl [  \tilde Y  \cdot \tilde S_j(X) \bigr ]  - \EE \bigl [ Y \cdot S_j(X) \bigr ] \Bigr |  \leq a_1 + a_2 \leq 2 M   \cdot \tau^{-2}
\$
for all $j \in [d]$. Thus choosing $\tau = 2   (M  \cdot n /\log d)^{1/4}$, we have 
\#\label{eq:grad5}
& \Bigl \| \EE \bigl [  \tilde Y  \cdot \tilde S_j(X) \bigr ]  - \EE \bigl [ Y \cdot S_j(X) \bigr ] \Bigr \| _{\infty} \notag \\
&\qquad  \leq 1/2      \cdot \sqrt{ M\cdot  \log d/n}.
\#
 Furthermore, under Assumption  \ref{assume:moments},  the variance of $\tilde Y \cdot \cdot \tilde S_j(X)$ is bounded by 
 \$
 \Var[ \tilde Y\cdot \tilde S_j(X) ] & \leq \EE [ \tilde Y^2 \cdot \tilde S_j^2 (X) ] \notag \\ 
 & \leq \EE  [ Y^2 \cdot  S_j^2 (X) ] \notag \\ 
 &\leq \sqrt{ \EE (Y^4) \cdot \EE [ S^4_j (X) ] } \leq M.
 \$
Thus for the second term in \eqref{eq:grad1}, since $|\tilde Y \cdot \tilde S_j (X) | \leq \tau^2$, by the Bernstein  inequality in \cite{boucheron2013concentration} (Theorem 2.10), for any $j \in [d]$ and any $t >0$, 
we have 
\#\label{eq:concentration1}
& \PP \biggl \{ \biggl |\frac{1}{n}  \sum_{i=1}^n  \tilde Y_i \cdot \tilde S_j (X_i)  - \EE \bigl [ \tilde Y \cdot \tilde S_j(X) \bigr ] \biggr | \notag \\
& \qquad \qquad \geq \sqrt{ \frac{ 2M \cdot t}{n} } + \frac{\tau^2 \cdot t }{3n}\biggr \}  \leq \exp (-t) .
\#
Taking union bound over $j \in [t]$ in  \eqref{eq:concentration1} yields 
\#\label{eq:concentration2}
& \PP \biggl \{ \biggl \|\frac{1}{n}  \sum_{i=1}^n  \tilde Y_i \cdot \tilde S_j (X_i)  - \EE \bigl [ \tilde Y \cdot \tilde S_j(X) \bigr ] \biggr \|_{\infty} \\
&\qquad   \qquad  \geq \sqrt{ \frac{ 2 M \cdot t}{n} } + \frac{\tau^2 \cdot t }{3n}\biggr \}  \leq \exp (-t + \log d). \notag
\#
Finally, we plug in $\tau = 2 ( M \cdot  n /  \log d )^{1/4} $  and set $t =3  \log d $ in  \eqref{eq:concentration2} to obtain that 
\#\label{eq:grad6}
& \biggl \|\frac{1}{n}  \sum_{i=1}^n  \tilde Y_i \cdot \tilde S_j (X_i)  - \EE \bigl [ \tilde Y \cdot \tilde S_j(X) \bigr ] \biggr \|_{\infty} \notag \\
 & \qquad \qquad \leq  (4 +\sqrt{6})   \sqrt{ \frac{ M \cdot \log d}{n} }
\#
with probability at least $1 - d^{-2}$. Finally, combining \eqref{eq:grad1}, \eqref{eq:grad5}, and \eqref{eq:grad6}, we conclude the proof.
\end{proof}

\subsection{Proof of Lemma \ref{lem:grad_loss2}}\label{sec:proof:lemmaLgrad_loss2}
\begin{proof} For loss function $L$ defined in \eqref{eq:estimator} in the matrix setting, we have
\#\label{eq:compute_grad}
 & \nabla L(\mu \beta^*)  =  2 \mu \beta^* - \frac{2} {\kappa \cdot n} \sum_{i=1}^n    \psi \bigl [ \kappa \cdot  Y_i \cdot     S(X_i) \bigr ]\notag \\ 
 &
 \qquad = 2 \EE [ Y \cdot S(X) ] -  \frac{2} {\kappa \cdot n} \sum_{i=1}^n  \psi \bigl [ \kappa \cdot  Y_i \cdot     S(X_i) \bigr ] . 
 \#
Here the last equality follows from the generalized Stein's identity.  In the sequel, we apply results in  \cite{minsker2016sub} to bound $ \| \nabla L(\mu \beta^*)\|_{\oper}$. To begin with, we first consider the operator norm of $\EE [ Y^2 \cdot S (X) S(X) ^{\top }]  \in \RR^{d_! \times d_2}$ and $\EE [ Y^2\cdot  S(X) ^\top S(X) ] \in \RR^{d_2 \times d_2}$.  For notational simplicity, we denote by $S_{j ,\cdot} (\cdot)\in \RR^{d_2}$  $S_{\cdot , k}  (\cdot) \in \RR^{d_1}$  the $j$-th row and $k$-the column of the score function $S(\cdot)$, respectively. For any $u \in \cS^{d_1 -1}$, by Cauchy-Schwarz inequality we have 
\#\label{eq:op1}
& \EE [ Y^2 \cdot  u ^\top S (X) S(X) ^\top u ]  \notag  \\
 & \qquad 
 = \sum_{k=1}^{d_2}  \EE \bigl \{  [ Y^2 \cdot S_{\cdot ,k }(X) ^\top u ]^2 \bigr \} \notag \\
 & \qquad  \leq d_2 \cdot  \sqrt{ \EE (Y^4) \cdot  \EE  \bigl \{   [ S_{\cdot, 1} (X)^\top u  ]^4   \bigr \} } ,
\#
where we use the fact that the entries of $S(X)$ are i.i.d. Since $\EE [ S_{ij} (X) ]= 0$ and $\EE [ S_{ij}^4 (X) ] \leq M$, by Cauchy-Schwarz inequality we obtain that 
\#\label{eq:op2}
& \EE  \bigl \{   [ S_{\cdot, 1} (X)^\top u  ]^4   \bigr \}  \notag \\
&\qquad  = \sum_{j _1=1}^d \sum_{j_2=1}^d \EE [S_{j_1, 1}(X)^2 \cdot S_{j_2 ,1} ^2 (X) ] \cdot   u_{j_i}^2 u_{j_2 }^2  \notag \\
& \qquad  \leq \sum_{j _1=1}^d \sum_{j_2=1}^d  \sqrt{ \EE [ S_{j_1, 1}^4 (X) ] \cdot \EE [ S_{j_2, 1}^4 (X) ] }\cdot    u_{j_i}^2 u_{j_2 }^2  \notag \\
&\qquad \leq M  \sum_{j _1=1}^d \sum_{j_2=1}^d   u_{j_i}^2 u_{j_2 }^2  = M.
\#
Thus combining \eqref{eq:op1} and \eqref{eq:op2} we obtain that 
\$
\EE [ Y^2 \cdot  u ^\top S (X) S(X) ^\top u ] \leq d_2 \cdot M,
\$
  which implies that $  \| \EE [ Y^2 \cdot    S (X) S(X) ^\top   ] \|_{\oper} \leq d_2 \cdot M$. Similarly, we  obtain~$ \| \EE [ Y^2 \cdot    S (X) ^\top S(X)   ] \|_{\oper} \leq  d_1 \cdot M.$
Thus by Corollary 3.1 in \cite{minsker2016sub}, we have
\#\label{eq:op3}
&\PP \biggl  \{   \biggl \| \frac{1 } {\kappa \cdot n} \sum_{i=1}^n  \psi \bigl [ \kappa \cdot  Y_i \cdot     S(X_i) \bigr ] - \EE [ Y \cdot S(X) ] \biggr \|_{\oper} \geq  \frac{t} {\sqrt{n}} \biggr \}  \notag \\
&\qquad  \leq 2 (d_1 + d_2) \exp\bigl [ - \kappa t   \sqrt{n}  + \kappa^2 (d_1 + d_2) M /2 \bigr ]
\#
for any $ t >0$ and $\kappa >0$. We set $$\kappa = 2  \sqrt{n \cdot \log (d_1 + d_2)}  /\sqrt{  (d_1 + d_2 )M}$$  and $t = \sqrt{ (d_1 + d_2 ) M } \cdot s $ in \eqref{eq:op3}, which implies that 
\#\label{eq:op4}
& \PP \biggl  \{   \biggl \| \frac{1 } {\kappa \cdot n} \sum_{i=1}^n   \psi \bigl [ \kappa \cdot  Y_i \cdot     S(X_i) \bigr ] - \EE [ Y \cdot S(X) ] \biggr \|_{\oper} \notag \\
&\qquad \qquad \geq   \sqrt{ \frac{ (d_1+d_2) M}{n} } \cdot s \biggr \}  \notag \\
& \qquad \leq 2 (d_1 + d_2 ) \cdot \exp \bigl [ -   2 \sqrt{ \log (d_1 + d_2) } \cdot s   \notag \\
& \qquad \qquad +2  \cdot \log  (d_1+d_2)  \bigr ].
\#
Now we   set  $s = 3  \cdot \sqrt{ \log (d_1 + d_2)}$, which implies that  the right-hand side of \eqref{eq:op4} is less than \$
&2 (d_1+ d_2) \cdot  \exp \bigl [ -   6  \log (d_1 + d_2)     + 2  \cdot \log  (d_1+d_2)  \bigr ] \\
&\quad \leq (d_1+ d_2)^2 \cdot \exp \bigl [ -  4 \cdot  \log (d_1 + d_2)      \bigr ]  = (d_1+ d_2)^{-2}.
\$
Therefore, combining \eqref{eq:compute_grad}
 and \eqref{eq:op4} we obtain  that 
 \$
 \| \nabla L (\mu \beta ^* ) \|_{\oper} \leq 6   \sqrt{ (d_1+ d_2)  \cdot M/ n} \cdot \sqrt{ \log (d_1 + d_2)} 
 \$
with probability at least $1 - (d_1 + d_2)^{-2}$, which  concludes the proof.

  \end{proof}


\subsection{Proof of Lemma \ref{lemma:truncation2}} \label{proof:lemma:truncation2}

\begin{proof}

By triangle inequailty, we have
\#\label{eq:triangle}
\| \tilde \Sigma - \Sigma^* \|_{\infty} \leq  \| \tilde \Sigma - \EE   \tilde \Sigma   \|_{\infty } + \| \EE \tilde \Sigma - \Sigma ^* \|_{\infty }.
\#
	In the sequel, we bound the second term on the right-hand side of \eqref{eq:triangle}, which controls the bias of truncation.
	For each $j, k \in [d]$, we have
	\#\label{eq:bias_3terms}
	& \left |    \EE \tilde \Sigma_{jk} - \Sigma ^*_{jk} \right | \leq  \left | \EE  \bigl [  \tilde Y\cdot  \tilde T _{jk}(X) \bigr ]  - \EE \bigl [   Y \cdot   T_{jk}(X) \bigr ]  \right |  \notag \\
	&  \qquad  \leq \left |  \EE \bigl \{   \tilde Y   \cdot \bigl [ \tilde T _{jk}(X)  -   T_{jk}(X)\bigr ] \bigr \}  \right | \notag \\
	&\qquad \qquad + \left | \EE \bigl [  ( \tilde Y  - Y )  \cdot T_{jk}(X) \bigr ]\right | .
	\#
	For the first term in \eqref{eq:bias_3terms}, note that \$\tilde T_{jk} (X) -   T_{jk} (X)  =   T_{jk} (X)  \cdot \ind \{ |T_{jk} (X) | \geq \tau^2 \}. \$   Then by Cauchy-Schwarz inequality we have
	\#\label{eq:bias11}
	& \left |  \EE \left \{  \tilde Y \cdot    \left [   \tilde T_{jk}(X) -   T_{jk}(X) \right ] \right \} \right | ^2 \notag \\
	& \qquad = \left | \EE \left [ \tilde Y \cdot  T _{jk}(X) \cdot \ind \{ | T_{jk}(X)| \geq \tau^2  \} \right ] \right|^2 \notag \\
	&\qquad \leq \EE   \left [  \tilde Y^2 \cdot T_{jk}^2(X)   \right ] \cdot \PP  \left [  | T_{jk}(X) | \geq \tau ^2  \right ] .
	\#
  Furthermore, by H\"older's inequality, we have
  \#\label{eq:bias11a}
& \EE \left [ \tilde Y ^2 \cdot  T_{jk}^2 (X) \right ] \leq \left [\EE (\tilde Y^6) \right ]^{1/3} \cdot \left \{  \EE \left [ | T_{jk}(X) | ^3 \right ]  \right \}^{2/3} \notag \\
&\qquad  \leq \left [\EE (  Y^6) \right ]^{1/3} \left \{  \EE  \left [ | T_{jk}^3 (X) |     \right ]  \right \}^{2/3} .
  \#
	If $j \neq k$, by the definition of $T(x)$  in \eqref{eq:second_score}, we have $T_{jk} (x) = S_j(x) \cdot S_k(x)$, $\forall x \in \RR^d$.
	Then by Cauchy-Schwarz inequality, we have
	\#\label{eq:bound_tjk1}
	& \EE \left [ | T_{jk}^3 (X) |  \right ]    = \EE \left [ | S_j(X)|^3 \cdot | S_k(X) |^3 \right  ]  \notag \\ & \qquad 
	\leq  \sqrt{   \EE    [ S_j^6(X)    ] \cdot \EE  [ S_k^6 (X)   ]  }  = \EE    [ S_j^6(X)    ] .
	\#
	In addition, if $j = k$, by \eqref{eq:second_score}, $T_{jj}(x) = S_j^2 (x) - s_1 (x_j) $. Since $(a + b)^3 \leq 4 (a^3 + b^3)$ for any $a, b > 0$, we have
	\#\label{eq:bound_tjk2}
	\EE \left [ | T_{jj}^3(X) |   \right ] \leq 4 \EE [ S_j^6(X) ] + 4 \EE \left[ |s_1 ^3 (X_j)|  \right] .
	\#
	Moreover, by \eqref{eq:bias11}, \eqref{eq:bias11a}, and  the Markov's inequality that 
	 \$
	\PP  \left [  | T_{jk}(X) | \geq \tau ^2  \right ] \leq \EE \left [ | T_{jk}^3(X) |   \right  ]\cdot \tau^{-6} ,
	 \$ 
    we further have
   \#\label{eq:bias11b}
 & \left |  \EE \bigl \{   \tilde Y   \cdot \bigl [ \tilde T _{jk}(X)  -   T_{jk}(X)\bigr ] \bigr \}  \right | ^2 \notag \\
 &\qquad  \leq \left [ \EE ( Y^6)  \right ]^{1/3} \cdot \left \{ \EE \left [ | T_{jk}^3(X) |   \right  ] \right \}^{5/3}   \cdot \tau^{- 6}  \notag \\
 &\qquad \leq 32 M^2 \cdot \tau^{-6} .
   \#
Here the last inequality follows from combining Assumption \ref{assume:moments}, \eqref{eq:bound_tjk1}, and \eqref{eq:bound_tjk2}.

	Similarly, for the second term in \eqref{eq:bias_3terms}, by the H\"older's inequality and the Markov's inequality we obtain that 
	\#\label{eq:bias12}
	&\left |\EE \left [  ( \tilde Y  - Y ) \cdot T_{jk}(X)    \right ] \right | ^2 \notag \\
	&\qquad  \leq  \left [ \EE    (Y^6)\right ] ^{1/3} \cdot \left \{ \EE \left [  |T_{jk}^3 (X) |\right] \right \}^{2/3} \cdot  \PP ( | Y | \geq \tau ) \notag\\
		& \qquad  \leq \left [ \EE    (Y^6)\right ] ^{4/3}\cdot  \left \{ \EE \left [  |T_{jk}^3 (X) |\right] \right \}^{2/3} \cdot \tau^{-6} \notag \\
		&\qquad  \leq 4 M^2 \cdot \tau^{-6} .
	\#
Thus, combining \eqref{eq:bias_3terms}, \eqref{eq:bias11b}, and  \eqref{eq:bias12}, we~have 
\#\label{eq:bias_final}
\| \EE \tilde \Sigma - \Sigma ^* \|_{\infty} \leq 9  M \cdot \tau^{-3}.
\#

In what follows, we give a high-probability bound on $\| \tilde \Sigma - \EE \tilde \Sigma \|_{\infty}$ using concentration inequalities, which combined with \ref{eq:bias_final}, concludes the proof.

 For any $j,k \in [d], $ note that $ | \tilde Y \cdot \tilde T_{jk }(X) | \leq \tau^3 $. In addition, by assumption \ref{assume:moments}, its  variance
is bounded by
\$
& \Var\left [ \tilde Y \cdot \tilde T_{jk }(X) \right ] \leq \EE \left [  Y ^2 \cdot T_{jk}^2 (X) \right] \notag \\
&\qquad  \leq  \left [  \EE ( Y^6) \right ]^{1/3}  \cdot \left \{ \EE \left [  |T_{jk}^3 (X) |\right] \right \}^{2/3} \leq 2 M.
\$
Now we apply the Bernstein's inequality \citep{boucheron2013concentration} (Theorem 2.10)  to $\{ \tilde Y _i\cdot \tilde T_{jk}(X_i)\}_{i\in[n]}$ and obtain that
\#\label{eq:tail_trunc}
& \PP \biggl  \{ \biggl | \frac{1}{n} \sum_{i=1}^n \tilde Y _i \cdot \tilde T_{jk}(X_i)  - \EE \left  [ \tilde Y \cdot \tilde T_{jk} (X)   \right ]\biggr  |  \notag \\
&\qquad    \geq  \sqrt{ \frac{4 M\cdot t }{n} }+ \frac{ \tau^3 \cdot t }{ 3n}  \biggr   \} \leq 2 \exp (-t).
\#
Taking a union bound over $j, k \in [d]$ in \eqref{eq:tail_trunc}, we obtain that
\#\label{eq:union_bound}
&  \PP \left [  \| \tilde \Sigma - \EE \tilde \Sigma \|_{\infty} \geq \sqrt{  4M\cdot t /n   }+  \tau^3 \cdot t /  (3n)    \right ] \notag \\ 
& \qquad  \leq 2 \exp ( -t + 2 \log d ).
\#
Choosing $t = 5 \log d$ in \eqref{eq:union_bound}, we obtain that 
\#\label{eq:concentration_final}
& \| \tilde \Sigma - \EE \tilde \Sigma  \|_{\infty} \notag \\
&\qquad \leq 2   \sqrt{  5M \log d /n} +2  \tau^{3} \cdot  \log d /n
\#
 holds with probability at least $1 - d^{-2}$. Finally, combining \eqref{eq:bias_final} and \eqref{eq:concentration_final}, we complete the proof of Lemma~\ref{lemma:truncation2}.
\end{proof}

\end{document}